\documentstyle[12pt]{article}

\font\fraktwelve=eufm10 at 12pt
\font\frakten=eufm10
\font\frakseven=eufm7

\def\eul{\fam=9}

\newcommand{\Gp}{{\eul p}}
\newcommand{\Gr}{{\eul r}}

\newcommand{\Gf}{{\eul f}}

\catcode`\@=11

\def\section{\@startsection {section}{1}{\z@}{-3.5ex plus
  -1ex minus -.2ex}{2.3ex plus .2ex}{\normalsize\bf}}
\def\subsection{\@startsection {subsection}{2}{\z@}{-3.25ex plus
  -1ex minus -.2ex}{1.5ex plus .2ex}{\normalsize\bf}}

\catcode`\@=12

\newtheorem{thm}{Theorem}[section]
\newtheorem{prop}[thm]{Proposition}

\newtheorem{lemma}[thm]{Lemma}

\newcommand{\blacksquare}{{\vrule height2ex width0.75em}}

\newenvironment{proof}{\begin{sc}\noindent Proof: \end{sc}}{
     \hbox to 2em{}\nobreak\hfill$\blacksquare$\par\medskip}

\newcommand{\Gal}{\mathop{\rm Gal}}
\newcommand{\Hom}{\mathop{\rm Hom}}

\begin{document}


\begin{center}
Class groups and Selmer groups \\
Journal of Number Theory, (56), 1996, 79 - 114
\end{center}
\vspace{.4in}

\begin{center}
Edward F. Schaefer, Santa Clara University\footnote{{\bf Acknowledgments:} I
thank Hendrik Lenstra for his support and guidance throughout the preparation
of this paper and for suggesting the proofs of lemma 5.1 and theorem 5.3.
I also thank Bas Edixhoven, Everett Howe,
and Joseph Wetherell for their helpful advice. I also thank Kestutis \v{C}esnavi\v{c}ius
for pointing out an error in the published article - it has been fixed in this version.}
\end{center}
\vspace{.4in}

\noindent
{\bf Abstract.}
It is often the case that a Selmer group of an abelian variety and a group
related to an ideal class group can both be naturally embedded into the same
cohomology group. One hopes to compute one from the other by finding how close
each is to their intersection. In this paper we compute the two groups and
their intersection explicitly in the local case and put together the local
information to get sharp upper bounds in the global case.
The techniques in this
paper can be used for arbitrary abelian varieties,
isogenies and number fields
assuming a frequently occurring condition.
Several examples are worked out for the Jacobians of elliptic and
hyperelliptic curves.

\section{Introduction}

It has been known for some time that a Selmer group
of an abelian variety
and a group related to an ideal class group can often be embedded
into the same cohomology group.
Ideally, the images of the two groups are almost the same and
we can compute one from the other.
In order to do this we need to
find how close the intersection is to each of the groups.
In the past, this has been done for special cases of
elliptic curves and with the Jacobians of Fermat curves.
Using techniques from Galois cohomology, we will see that we can do this
in much greater generality.

Let $A$ and $A'$ be abelian varieties of the same dimension $g$,
defined over $K$, an algebraic
number field, and let $\phi$ be a $K$-defined isogeny from
$A$ onto $A'$ (see Cornell and Silverman \cite{CS},
Lang \cite{La2} or Mumford \cite{Mu}
as general references on abelian varieties and Silverman \cite{Si}
for elliptic
curves).
The field $L$, which is the minimal field of definition of the points
in $A[\phi ]$, the kernel of $\phi$, is Galois over $K$.
If $H^{1}(\Gal (L/K),A[\phi ])$ is trivial, then the Selmer group
$S^{\phi}(K,A)$
can be embedded into the group $\Hom_{G(L/K)}(\Gal (\overline{L}/L),A[\phi ])$
(see section~\ref{clgpselgp} for a description of these objects).
Consider the subgroup of the homomorphism group of homomorphisms that
factor through the Galois group of a totally unramified extension of $L$;
call this group $C^{\phi}(K,A)$. From
class field theory, this group is related to the
ideal class group of $L$.
Denote by $I^{\phi}(K,A)$ the intersection of the group
$C^{\phi}(K,A)$ and the Selmer group. We would like to compute
$S^{\phi}(K,A)/I^{\phi}(K,A)$ and $C^{\phi}(K,A)/I^{\phi}(K,A)$ as accurately
as possible so as to relate the groups $S^{\phi}(K,A)$ and $C^{\phi}(K,A)$.
We do this with local computations.

Let $\Gp$ be a prime of $K$ and $K_{\Gp}$ be the completion of $K$ at
that prime.
Denote by $S^{\phi}(K_{\Gp},A)$ the group $A'(K_{\Gp})/\phi A(K_{\Gp})$.
We shall assume that $H^{1}(G,A[\phi ])$ is trivial for all
$G\subseteq \Gal(L/K)$.
This condition occurs frequently and we will see some examples
in sections 5.
Under this assumption we can embed $S^{\phi}(K_{\Gp},A)$ into the
group
$\Hom_{G(LK_{\Gp}/K_{\Gp})}(\Gal (\overline{LK_{\Gp}}/LK_{\Gp}),A[\phi ])$.
Let $C^{\phi}(K_{\Gp},A)$ be the subgroup of the
homomorphism group of homomorphisms
that factor through the
Galois group of an unramified extension of $LK_{\Gp}$.
Denote by
$I^{\phi}(K_{\Gp},A)$ the intersection of the groups
$C^{\phi}(K_{\Gp},A)$ and $S^{\phi}(K_{\Gp},A)$.

Under the assumptions that
$H^{1}(G,A[\phi ])=H^2(G,A[\phi])=0$ for all
$G\subseteq \Gal (L/K)$;
we have the following injections of groups
\[
S^{\phi}(K,A)/I^{\phi}(K,A)\hookrightarrow
\prod_{\Gp}S^{\phi}(K_{\Gp},A)/I^{\phi}(K_{\Gp},A)\]
\[
C^{\phi}(K,A)/I^{\phi}(K,A)\hookrightarrow
\prod_{\Gp}C^{\phi}(K_{\Gp},A)/I^{\phi}(K_{\Gp},A)\]
where $\Gp$ ranges over the primes of $K$. We will show in
section~\ref{glob comp} that
for each finite prime $\Gp$ of $K$ that does not divide the conductor of
$A$ or the degree of $\phi$,
the groups $S^{\phi}(K_{\Gp},A)$, $I^{\phi}(K_{\Gp},A)$,
$C^{\phi}(K_{\Gp},A)$ and $A(K_{\Gp})[\phi ]$
are isomorphic.

We will also see along the way that the orders of the groups
$S^{\phi}(K_{\Gp},A)$, $C^{\phi}(K_{\Gp},A)$,
and $I^{\phi}(K_{\Gp},A)$ can often be easily computed.
Computing the sizes of $S^{\phi}(K_{\Gp},A)$ and $C^{\phi}(K_{\Gp},A)$
is standard. The substantial contribution here is the computation
of $I^{\phi}(K_{\Gp},A)$ where $\Gp$ is a finite prime that divides
the conductor of $A$.
This is the computation that has caused others trouble in the past.
For elliptic
curves, these computations are greatly facilitated by the algorithm of
Tate \cite{Ta}.
These injections then give us upper bounds on the index of the
group $I^{\phi}(K,A)$ in the Selmer group and the group related
to the ideal class group. We are most interested
when the intersection is close to each of the
groups,
because then
we can often compute one group from the other or at least bound the size of
one given the size of the other.

Selmer groups hold key information about the group of rational points of an
abelian variety over an algebraic number field, known as the Mordell-Weil
group.
Since neither Mordell-Weil groups nor class groups are well understood
currently, any connection between them is helpful.
On the lighter side,
part of the lore of number theory is that
high-rank elliptic curves can produce record-breaking class groups.
The practicalities of this were first worked out by Mestre \cite{Me-1}.
Several papers have
appeared recently where techniques were presented for computing one of these
groups from the other.
Eisenbeis, Frey and Ommerborn \cite{Fr}
studied elliptic curves of the form $Y^{2}=X^{3}+k$ over {\bf Q} with
the multiplication by 2 map (from now on the 2-map) and the 2-rank of
the class groups of pure cubic fields.
The same curves have a rational 3-isogeny, and Satg\'{e} \cite{Sa}
studied these
and the 3-rank of the class groups of quadratic fields.
Washington \cite{Wa} studied the 2-map for curves of the
form $Y^{2}=X^{3}+mX^{2}-(m+3)X+1$ over ${\bf Q}$ and the 2-rank of the
class groups of the simplest cubic
fields (see Cohn \cite{Co} or Shanks \cite{Sh}).
Brumer and Kramer \cite{BK} produced techniques for computing the connection
in a more general domain.  They studied the 2-map and
cubic extensions for elliptic curves,
defined over number fields that are unramified over ${\bf Q}$ at 2,
at primes of good or multiplicative reduction.
McCallum \cite{Mc} has studied the Jacobians of the $p$th Fermat
curves and quotients of Fermat curves
using $p$-isogenies and the $p$-part
of the class group of ${\bf Q}(\zeta_{p})$.

In the following we produce techniques that work in far greater
generality. In section~\ref{clgpselgp}
we give a
global description of the Selmer group and the group related to an
ideal class group. In section~\ref{localcomp}
we present results that are useful for
computing
these groups and their
intersection over local fields.
In section~\ref{glob comp} we put these local results together
to prove a theorem giving upper bounds for the index of the
global intersection in the Selmer group and in the group related to an
ideal class group.
In section~\ref{2-map} we closely analyze the
2-map. We will first give criteria for the cohomological triviality
of $J[2]$ as a $\Gal(L/K)$-module where $J$ is the Jacobian of a
hyperelliptic curve. When $J[2]$ is cohomologically trivial we can
relate $C^{2}(K,J)$ to the 2-parts of the class groups of a collection
of subfields of $L$.
We finish this section by presenting
three examples. In the first example we show how to use
a curve of Mestre's to produce
a non-cyclic cubic extension of the rationals whose class group has 2-rank at
least 13. Then we look at the Jacobian of a hyperelliptic curve
whose 2-torsion is defined over a simplest quintic field and whose
Mordell-Weil rank over the rationals is 6 or 7. Lastly, we
see that the group related to the ideal class group is not always
contained in the Selmer group as it seems to be in all published
examples.

\section{Selmer groups and class groups}
\label{clgpselgp}
Let $K$ be a number field; this shall always mean that
$K$ is a finite extension of ${\bf Q}$.
Let $A$ and $A'$ be abelian varieties defined over $K$
of the same dimension $g$ and let $\phi$
be a $K$-defined isogeny of $A$ onto $A'$.
Let $\Gal (\overline{K}/K)$ denote the absolute Galois group of $K$.
By taking $\Gal (\overline{K}/K)$-invariants of the groups in
the short exact sequence
\[0 \rightarrow A[\phi ] \rightarrow
A(\overline{K}) \stackrel{\phi}{\rightarrow}
A'(\overline{K})
\rightarrow 0\]
we obtain the following long exact sequence.
The symbol
$[\phi ]$ after a group denotes
its subgroup sent to 0 by $\phi$, and the expression $H^{1}(K,M)$ denotes
$H^{1}(\Gal (\overline{K}/K),M)$ for some $\Gal (\overline{K}/K)$-module $M$.
\[
0\rightarrow A(K)[\phi ]\rightarrow A(K)\stackrel{\phi}{\rightarrow}
A'(K)\rightarrow H^{1}(K,A[\phi ])\rightarrow H^{1}(K,A(\overline{K}))
\stackrel{\phi}{\rightarrow}\ldots\]
This gives us the short exact sequence
\[
0\rightarrow A'(K)/\phi A(K)\rightarrow H^{1}(K,
A[\phi ])\rightarrow H^{1}(K,A(\overline{K}))[\phi ]\rightarrow 0.
\]
See Atiyah and Wall \cite{AW}
as a reference on group cohomology.

Let $\Gp$ be a prime of $K$, finite or infinite, and
$K_{\Gp}$ be the completion of $K$ at the prime $\Gp$.
For each prime $\Gp$
of $K$ we obtain similar
short exact sequences
\[
0 \rightarrow A[\phi ] \rightarrow
A(\overline{K}_{\Gp}) \stackrel{\phi}{\rightarrow}
A'(\overline{K}_{\Gp})
\rightarrow 0\]
and
\[
0\rightarrow A'(K_{\Gp})/\phi A(K_{\Gp})\rightarrow H^{1}(K_{\Gp},
A[\phi ])\rightarrow H^{1}(K_{\Gp},A(\overline{K_{\Gp}}))[\phi ]\rightarrow 0.
\]

Since the group $\Gal (\overline{K_{\Gp}}/K_{\Gp})$ is isomorphic to a
decomposition group for $\Gp$ in $\Gal (\overline{K}/K)$, we can
restrict cocycles in $\Gal (\overline{K}/K)$ to cocycles in
$\Gal (\overline{K_{\Gp}}/K_{\Gp})$. This induces the restriction map
\[
H^{1}(K,A[\phi ])\rightarrow H^{1}(K_{\Gp},A[\phi ]).\]

We present the following commutative diagram in order to define the
$\phi$-Selmer group for $A$ over $K$
\[
\begin{array}{rcccccccl}
0&\rightarrow&A'(K)/\phi A(K)&\rightarrow&
H^{1}(K,A[\phi ])&\rightarrow&
H^{1}(K,A(\overline{K}))[\phi ]&\rightarrow&0 \\
&&\downarrow& &\downarrow&\searrow \beta&\downarrow&&\\
0&\rightarrow&\prod\limits_{\Gp}A'(K_{\Gp})/\phi A(K_{\Gp})&
\rightarrow&
\prod\limits_{\Gp}H^{1}(K_{\Gp},A[\phi ])&
\rightarrow &\prod\limits_{\Gp}
H^{1}(K_{\Gp},A(\overline{K}_{\Gp}))[\phi ]&
\rightarrow&0
\end{array}
 \]
where $\Gp$ ranges over the primes of $K$.
The $\phi$-Selmer group for $A$ over $K$ is the kernel of $\beta$ and
is denoted $S^{\phi}(K,A)$. The Selmer group clearly contains
$A'(K)/\phi A(K)$.

The group of rational points $A(K)$,
is often called the Mordell-Weil group; it is a finitely generated
abelian group. Since the torsion of an abelian variety is computeable,
one could find the free ${\bf Z}$-rank of $A(K)$ if one
could compute the group $A(K)/nA(K)$ for some integer $n$ greater than 1.
So one often takes $\phi$ to be a multiplication by $n$ map.
If $\phi$ is a map from $A$ to $A'$ and $\phi '$ is a map from $A'$ to $A$
with the property that $\phi '\circ\phi =[n]$ then
the following is an exact sequence of groups
\[ 0\rightarrow \frac{A'(K)[\phi ']}{\phi (A(K)[n])}
\rightarrow \frac{A'(K)}{\phi A(K)}
\stackrel{\phi '}{\rightarrow}\frac{A(K)}{nA(K)}
\rightarrow \frac{A(K)}{\phi 'A'(K)}\rightarrow 0.\]
If one could compute
$A'(K)/\phi A(K)$
and $A(K)/\phi 'A'(K)$,
then one could compute \linebreak
$A(K)/nA(K)$.
There is no known algorithm to effectively compute the group
$A'(K)/\phi A(K)$ however.
A Selmer group is the closest known approximation to $A'(K)/\phi A(K)$
that is effectively computable, which
is why they were originally of interest.

The field
$L$, which is the minimal field of definition of the points in $A[\phi ]$,
is Galois over
$K$.
Restricting cocycles in $H^{1}(K,A[\phi ])$
to the subgroup $\Gal (\overline{L}/L)$ induces the
following exact sequence of groups, called an inflation-restriction
sequence
\[ 0\rightarrow H^{1}(\Gal (L/K),A[\phi ])
\rightarrow H^{1}(K,A[\phi ])\rightarrow
\Hom\nolimits_{G(L/K)}(\Gal (\overline{L}/L),A[\phi ])\]
where $\Hom_{G(L/K)}$ denotes the $\Gal (L/K)$-invariant homomorphisms.
Assuming that\linebreak
$H^{1}(\Gal (L/K),A[\phi ])$ is trivial, we
have
\[
H^{1}(K,A[\phi ])\hookrightarrow
\Hom\nolimits_{G(L/K)}(\Gal (\overline{L}/L),A[\phi ]).\]
We also assume that $H^{1}(G,A[\phi ])$
and $H^2(G,A[\phi])$ are trivial for all
groups $G$ contained in $\Gal (L/K)$, so
for each prime $\Gp$ of $K$ we also have
\[
H^{1}(K_{\Gp},A[\phi ])\hookrightarrow
\Hom\nolimits_{G(LK_{\Gp}/K_{\Gp})}(\Gal
(\overline{LK_{\Gp}}/LK_{\Gp}),A[\phi ]).\]

Let $m$ be the exponent of $A[\phi ]$ and
let ${\rm Cl}(L)$ denote the ideal class group of $L$. From
class field theory, the dual of
the group ${\rm Cl}(L)/{\rm Cl}(L)^{m}$ is naturally
isomorphic to the group\linebreak
$\Hom (\Gal (H(L)/L),{\bf Z}/m{\bf Z})$ where $H(L)$ is the maximal unramified
abelian extension of $L$, its Hilbert class field.
The group $C^{\phi}(K,A)$ is the subgroup of
$\Hom_{G(L/K)}(\Gal (\overline{L}/L),A[\phi ])$
of homomorphisms that factor through the Galois group
of a totally unramified extension of $L$.
This subgroup is isomorphic to the group
$\Hom_{G(L/K)}(\Gal (H(L)/L),A[\phi ])$
and so it is related to
the group ${\rm Cl}(L)/{\rm Cl}(L)^{m}$.
We will describe this relation for the 2-map and the Jacobians of
hyperelliptic curves in theorem~\ref{kernorm}.
We will call any homomorphism that
factors through the Galois group of a totally unramified extension an
unramified homomorphism.

The group $S^{\phi}(K,A)$
is the group of all elements of $\Hom_{G(L/K)}(\Gal (\overline{L}/L),A[\phi ])$
that map to the image of $A'(K_{\Gp})/\phi A(K_{\Gp})$ for all primes $\Gp$
of $K$.
The group $C^{\phi}(K,A)$
is the group
of all elements that map to unramified homomorphisms in
$\Hom_{G(LK_{\Gp}/K_{\Gp})}(\Gal (\overline{LK_{\Gp}}/LK_{\Gp}),A[\phi ])$
for all primes
$\Gp$ of $K$ since $L$ is normal over $K$.
These groups are the same locally for almost all primes, as we will
see in the next section.

\section{Local computations}
\label{localcomp}

In this section we describe
the group $A'(K_{\Gp})/\phi
A(K_{\Gp})$, the group of unramified homomorphisms in
$\Hom_{G(LK_{\Gp}/K_{\Gp})}(\Gal (\overline{LK_{\Gp}}/LK_{\Gp}),A[\phi ])$
and the subgroup of $A'(K_{\Gp})/\phi A(K_{\Gp})$ that maps to unramified
homomorphisms. If $H^{1}(G,A[\phi])=0$ for all $G\subseteq \Gal (L/K)$ then
these are the groups
$S^{\phi}(K_{\Gp},A)$, $C^{\phi}(K_{\Gp},A)$ and $I^{\phi}(K_{\Gp},A)$.
However, we will lift
the restriction that $H^{1}(G,A[\phi])=0$ for all $G\subseteq \Gal (L/K)$
until we discuss infinite primes.
Since all computations will be local, and
subscripts are annoying, we will omit them.
We deal with finite primes first.

Let $A$ and $A'$ be abelian varieties of the same dimension $g$ defined over
$K$,
the completion of a number field at a finite prime with residue characteristic
$p$.
Let $\phi$ be an isogeny from $A$ onto $A'$ that is defined over $K$.
Again define $L$ to be the field $K(A[\phi])$.

{}From
the short exact sequence
\[
0\rightarrow A[\phi ]\rightarrow A \rightarrow
A'\rightarrow 0\]
we obtain the following injection of groups
by taking $\Gal (\overline{K}/K)$-invariants
\[ A'(K)/\phi (A(K)) \hookrightarrow
H^{1}(K,A[\phi ]).
\] From
the restriction map from cohomology there is a homomorphism
\[
H^{1}(K,A[\phi ])
\rightarrow
\Hom\nolimits_{G(L/K)}(\Gal (\overline{L}/L),A[\phi ]).\]
We are first interested in describing the group of unramified homomorphisms
in\linebreak
$\Hom\nolimits_{G(L/K)}(\Gal (\overline{L}/L),A[\phi ])$ and the subgroup
of $A'(K)/\phi
A(K)$ that maps to those unramified homomorphisms.

\begin{lemma}
\label{le:unr hom}
The subgroup of unramified homomorphisms in
$\Hom_{G(L/K)}(\Gal (\overline{L}/L),A[\phi ])$
is isomorphic to the group
$A(K)[\phi ]$.
\end{lemma}
\begin{proof}
Let $L^{\rm unr}$ be the maximal unramified extension of $L$ in the
given algebraic closure.
The group $\Hom (\Gal (L^{\rm unr}/L),A[\phi ])$ is isomorphic
to the group $A[\phi ]$ by the map $f\mapsto f({\rm Frob}_{L})$.
This map respects the $\Gal (L/K)$-action on each module.
Taking $\Gal (L/K)$-invariants of both
groups provides the result.
\end{proof}

Now let us consider the subgroup of $A'(K)/\phi A(K)$ that maps to unramified
homomorphisms. We still need not assume that $H^{1}(G,A[\phi])=0$ for
$G\subseteq \Gal (L/K)$.
The connecting homomorphism from cohomology sends
$P\in A'(K)$ to the class of cocycles in $H^{1}(K,A[\phi ])$
consisting of cocycles $\xi$ where
$\xi (\sigma )=\sigma Q-Q$ for $\sigma \in
\Gal (\overline{K}/K)$ and $Q\in A$ where $\phi Q=P$.
The point $P$ maps to an
unramified homomorphism if and only if its inverse images under the
map $\phi$ are all defined over an unramified extension of $L$.
Since $L=K(A[\phi ])$, if one is defined over an unramified extension of
$L$, then they all are.
Let $A(K)[\phi ]$ have exponent $m$. From the
previous lemma, it is clear that if $P$ maps to an
unramified homomorphism,
then these preimages will be defined over the unramified extension of $L$
of degree $m$.

Let $L'$ be the maximal unramified subextension of $L$ over $K$.
Let $M$ be
the unramified extension of $L'$ of degree $m$ and let $\tau$ generate
$\Gal (M/K)$. The field $ML$ is the unramified extension of $L$ of
degree $m$.

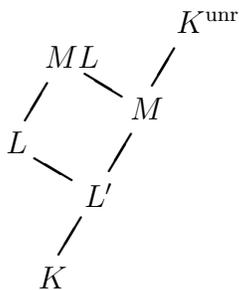
\begin{picture}(2000,1100)
\put(179,65){$K$}
\put(328,330){$L'$}
\put(476,610){$M$}
\put(76,502){$L$}
\put(200,778){$ML$}
\put(624,890){$K^{\rm unr}$}
\thicklines
\put(245,175){\line(3,5){72}}
\put(407,445){\line(3,5){72}}
\put(137,607){\line(3,5){72}}
\put(546,720){\line(3,5){72}}
\put(167,487){\line(5,-3){120}}
\put(329,757){\line(5,-3){120}}
\end{picture}

\begin{lemma}
\label{le:push ML to M}
Let the group $H^{1}(\Gal (L/L'),A[\phi ])$ be trivial.
Let $Q\in A(ML)$ and assume $\phi Q=P$, where $P\in A'(K)$.
There exists a $Q'\in A(M)$ such
that $\phi Q'=P$.
\end{lemma}
\begin{proof}
Since $H^{1}(\Gal (L/L'),A[\phi ])$ is trivial, so is
$H^{1}(\Gal (ML/M),A[\phi ])$. So
$H^{1}(M,A[\phi])$
injects into $H^{1}(ML,A[\phi])$ and so $A'(M)/\phi A(M)$ injects
into $A'(ML)/\phi A(ML)$. So if $P$ is in $\phi A(ML)$ then $P$ is
in $\phi A(M)$.
\end{proof}
\begin{lemma}
\label{ker norm}
The elements of $A(M)[\phi ]$ are in the kernel of the norm from $M$ to $K$.
\end{lemma}
\begin{proof}
If
$T \in A(M)[\phi ]$
then $T \in A(L')[\phi ]$. Since $N_{L'/K}(A(L')[\phi ])
\subseteq A(K)[\phi ]$,
it is clear that $N_{M/K}=m\cdot N_{L'/K}$ on $A(M)[\phi ]$.
Since $m$ is the exponent of $A(K)[\phi ]$,
all of $A(M)[\phi ]$ is in the kernel of $N_{M/K}$.
\end{proof}
\begin{thm}
\label{thm:mainthm}
Let $K$ be the completion of a number field at a finite prime and let
$\phi$ be a $K$-defined isogeny from $A$ to $A'$, which are
$K$-defined abelian varieties.
Let $L=K(A[\phi ])$ and let $L'$ be the maximal unramified
extension of $K$ contained in $L$. Let $m$ be the exponent of
$A(K)[\phi ]$, and $M$ be the unramified extension of $L'$ of degree $m$
and $\tau$ generate $\Gal (M/K)$.
Assume that the group $H^{1}(\Gal (L/L'),A[\phi ])$ is trivial.
The subgroup of $A'(K)/\phi A(K)$ that maps to unramified
homomorphisms in $\Hom_{G(L/K)}(\Gal (\overline{L}/L),A[\phi ])$
is isomorphic to the following group
\[ \frac{A(M)[\phi ]\cap (\tau -1)A(M)}
{(\tau -1)(A(M)[\phi ])}\]
by the map $P\mapsto\tau Q-Q$ where $P\in A'(K)$, $Q\in A(M)$ and
$\phi Q=P$.
\end{thm}
\begin{proof}
We chose $M$ to have degree $m$ over $L'$ so that $ML$ is the unramified
extension of $L$ of degree $m$.
A point $P$ in $A'(K)$ maps to the homomorphism $\sigma\mapsto\sigma
(\phi^{-1}(P))-\phi^{-1}(P)$ in $\Hom (\Gal (\overline{L}/L),A[\phi])$.
A point $P$ maps to an unramified homomorphism if the preimages, $\phi^{-1}
(P)$, are defined over $ML$.
From
lemma~\ref{le:push ML to M},
such a point has
a preimage defined over $M$.
Therefore the subgroup of $A'(K)/\phi A(K)$ that maps to unramified
homomorphisms is $(\phi A(M)\cap A'(K))/\phi A(K)$.
From the short exact sequence
\[
0\rightarrow A(M)[\phi]\rightarrow A(M)\stackrel{\phi}{\rightarrow}
\phi A(M)\rightarrow 0\]
we obtain the following exact sequence by taking
$\Gal (M/K)$-invariants
\[
0\rightarrow (\phi A(M)\cap A'(K))/
\phi A(K)\rightarrow H^{1}(\Gal (M/K),A(M)[\phi ])
\stackrel{\gamma}{\rightarrow} H^{1}(\Gal (M/K),A(M)).\]
Since $\Gal (M/K)$ is cyclic, generated by $\tau$, we have
\[
H^{1}(\Gal (M/K),A(M)[\phi ])\cong \frac{{\rm ker}\;
N_{\tau}:A(M)[\phi ]\rightarrow
A(M)[\phi ]}{(\tau -1)A(M)[\phi ]}.\] From
lemma~\ref{ker norm}, all of $A(M)[\phi ]$ is in the kernel of the norm.
The subgroup of $A'(K)/\phi A(K)$ that maps to unramified homomorphisms
is isomorphic to the kernel of $\gamma$ by the map
$P\mapsto\tau Q-Q$ where $P\in \phi A(M)\cap A'(K)$, $Q\in A(M)$ and
$\phi Q=P$.
The kernel of $\gamma$ is the group of elements that map to
coboundaries in $H^{1}(\Gal (M/K),A(M))$. So the kernel of $\gamma$ is
isomorphic to
\[ \frac{A(M)[\phi]\cap (\tau -1)A(M)}
{(\tau -1)(A(M)[\phi ])}.\]
\end{proof}

Let us refer to this group as the intersection quotient.
Because of the proof, the theorem holds when replacing $M$ by
any unramified extension of $M$.
Notice that if all of
$A(M)[\phi ]$ is contained in
$(\tau -1)(A(M))$, then
the intersection quotient
has the same order as $A(K)[\phi ]$
and so has the same order as the group of unramified homomorphisms.

Let us discuss the computation of
the intersection quotient. The group $A(M)[\phi ]$ and the
action of $\tau$ on it should be easy to find.
So we need to find which elements of $A(M)[\phi ]$ are in
$(\tau -1)(A(M))$.
The following
two lemmas brings this problem from the infinite to the finite as the
group $A(M)/A_{0}(M)$ is finite.
Let $NA$ be the N\'{e}ron model of $A$ over $\cal
O$, the ring of integers in $M$. Define $NA^{0}$ to be the open
subgroup scheme of $NA$ whose generic fiber is isomorphic to $A$ over
$M$ and whose special fiber is the identity component of the
closed fiber of $NA$. The group $NA^{0}({\cal O})$ is isomorphic to a
subgroup of $A(M)$ which we are denoting by $A_{0}(M)$;

\begin{lemma}
\label{le:ns coh tri}
The groups $H^{i}(\Gal (M/K),A_{0}(M))$ are trivial for all $i$.
\end{lemma}
\begin{proof}
In \cite{Mi},
Milne shows that if $W$ is the completion of a number field at a finite
prime, then
$H^{i}(\Gal (W^{\rm unr}/W),A_{0}
(W^{\rm unr}))$ is
trivial for all $i \geq 1$ where
$W^{\rm unr}$ is the maximal unramified
extension of $W$ in a given algebriac closure of $W$.
Since the N\'{e}ron model is stable under unramified base extensions
(see \cite{Ar}, p.\ 214), we know that since $M$ is an unramified
extension of $K$, that the set of
$\Gal (K^{\rm unr}/M)$-invariants of
$A_{0}(K^{\rm unr})$ is $A_{0}(M)$. This is not
necessarily the case otherwise.
Since $M^{\rm unr}=K^{\rm unr}$, we have
$H^{i}(\Gal (K^{\rm unr}/
M),A_{0}(K^{\rm unr}))=0$ for all $i \geq 1$. From
\cite{AW}, proposition 5, it follows that the sequence\[
0\rightarrow H^{i}(\Gal (M/K),A_{0}(M))
\stackrel{\rm inf}{\rightarrow}
H^{i}(\Gal (K^{\rm unr}/K),A_{0}(K^{\rm unr}))
\stackrel{\rm res}{\rightarrow}
H^{i}(\Gal (K^{\rm unr}/
M),A_{0}(K^{\rm unr}))\] is exact and we are done.
\end{proof}

\begin{lemma}
\label{le:EmodE0}
A point $T$ of $A(M)[\phi ]$
is in
$(\tau -1)(A(M))$ if and only if the image of $T$ in
$A(M)/A_{0}(M)$ is
in
$(\tau -1)(A(M)/A_{0}(M))$.
\end{lemma}
\begin{proof}
We have the following short exact sequence of $\Gal (M/K)$-modules
\[
0\rightarrow A_{0}(M)\rightarrow A(M)\rightarrow A(M)/A_{0}(M)\rightarrow 0.\]
By taking $\Gal (M/K)$-invariants, we see
from
lemma~\ref{le:ns coh tri} that\linebreak
$H^{i}(\Gal (M/K),A_{0}(M))$ is trivial
for all $i$. So we have
\[
H^{1}(\Gal (M/K),A(M))\cong H^{1}(\Gal (M/K),A(M)/A_{0}(M)).\]
Since $\Gal (M/K)$ is cyclic, each of these groups is isomorphic to the
kernel of the norm modulo the image of $(\tau - 1)$ on the appropriate
Galois-module.
From
lemma~\ref{ker norm}, the elements
of $A(M)[\phi ]$ are in the kernel of the norm from $M$ to $K$.
The result follows from the fact that the image of
$T$ is trivial in one group if and only if it is trivial in the other.
\end{proof}

We see that points in $A(M)[\phi ]$ (which equals $A(L')[\phi]$)
that are in $A_{0}(M)$ are
automatically in $(\tau -1)(A(M))$.  For points of $A(M)[\phi ]$ that
are not in $A_{0}(M)$ one must determine which have images that are in
$(\tau -1)(A(M)/A_{0}(M))$.

For elliptic curves,
the possible groups $E(M)/E_{0}(M)$ have been completely determined
by the classification of Kodaira \cite{Ko} and N\'{e}ron \cite{Ne}.
If we take for $E$ a minimal Weierstrass model, then
the group $E_{0}(M)$ is the subgroup of points of $E(M)$ with non-singular
reduction. In the following theorem we present an algorithm
for determining if a point $P$ in $E(M)[\phi]$ with
singular reduction
is in $(\tau - 1)(E(M))$. This is accomplished
in conjunction with
Tate's algorithm \cite{Ta}.
We must first determine whether or not $\tau$ acts non-trivially
on $E(M)/E_{0}(M)$.
If it does, then in most cases it is simply a matter of determining
whether or not $2|\# \Gal (L'/K)$ or
if the image of $P$ in $E(M)/E_{0}(M)\cong {\bf Z}/n{\bf Z}$ is even.

\begin{thm}
\label{thm:cases}
Let $E$ be defined by a minimal Weierstrass equation over $K$ and let
$\phi$ be a $K$-defined isogeny from $E$ onto $E'$. Let $L'=K(E[\phi ])\cap
K^{\rm unr}$, let $M$ be a finite
unramified extension of $L'$ and $\tau$ generate
$\Gal (M/K)$.
A point $P$ in $E(L')[\phi ]$
with singular reduction will be in $(\tau -1)(E(M))$ in exactly the
following cases. Otherwise it will not be in $(\tau -1)(E(M))$.
\begin{enumerate}
\item
$E$ has type $I_{\nu}$ reduction (multiplicative) with $\nu$ odd,
$\tau$ acts as $-1$ on $E(M)/E_{0}(M)$ and
$2|\#\Gal (L'/K)$.
\item
$E$ has type $I_{\nu}$ reduction with $\nu$ even,
$\tau$ acts as $-1$
on $E(M)/E_{0}(M)$
and the image of $P$ in
$E(M)/E_{0}(M)
\cong
{\bf Z}/\nu {\bf Z}$
is even.
\item
$E$ has type $IV$ or $IV^{\ast}$ reduction,
$\tau$ acts as $-1$
on $E(M)/E_{0}(M)$ and $2|\#\Gal (L'/K)$.
\item
$E$ has type $I_{\nu}^{\ast}$ reduction with $\nu$ odd, $\tau$ acts
as $-1$ on $E(M)/E_{0}(M)$ and the image of $P$ in $E(M)/E_{0}(M)
\cong {\bf Z}/4{\bf Z}$ is 2.
\item
$E$ has type $I_{\nu}^{\ast}$ reduction with $\nu$ even (including $\nu =0$),
and either the action of $\tau$ on $E(M)/E_{0}(M)$ has order 2 and
fixes $P$ modulo $E_{0}(M)$ or the action of $\tau$ on $E(M)/E_{0}(M)$
has order 3.
\end{enumerate}
\end{thm}
\begin{proof}
See N\'{e}ron \cite{Ne}
or Tate \cite{Ta} as a reference on the reduction types.
Since $E$ is given by a minimal Weierstrass equation, the points of
non-singular reduction are the same as the points of $E_{0}$.
Let $E$ have type $I_{\nu}$ reduction (multiplicative) over $K$.
The group
$E(K^{\rm{unr}})/E_{0}(K^{\rm{unr}})$ is isomorphic to the cyclic group of
$\nu$ components
of the special fiber of the minimal regular projective 2-dimensional
scheme $\cal E$ whose generic fiber is isomorphic over $K$ to $E$.
The scheme $\cal E$ is the minimal model for $E$ as a curve, not as
a group variety. The minimal model for $E$ as a group variety is the
N\'{e}ron minimal model (see \cite{Si}, p.\ 358).
The components
form a loop with each component crossing the two adjacent components.
The group $\Gal (K^{\rm{unr}}/K)$ acts on the
group of components and preserves
its structure. Therefore, the group $\Gal (K^{\rm{unr}}/K)$ must act as
$\pm 1$.
When $\Gal (K^{\rm unr}/K)$ acts as $+1$, then $E$ is said to have split
multiplicative reduction over $K$ and $E(K)/E_{0}(K)\cong {\bf Z}/\nu {\bf Z}$.
When $\Gal (K^{\rm unr}/K)$ acts as $-1$, then $E$ is said to have non-split
multiplicative reduction over $K$ and $E(K)/E_{0}(K)\cong {\bf Z}/2{\bf Z}$
when $\nu$ is even and $E(K)/E_{0}(K)$ is trivial when $\nu$ is odd.
If $\nu$ is odd and
$E(K)/E_{0}(K)\neq 0$ then $E(K)/E_{0}(K)\cong {\bf Z}/\nu {\bf Z}
\cong E(M)/E_{0}(M)$
and so $\tau$ acts trivially on $E(M)/E_{0}(M)$. Therefore $P$ cannot
be in the image of $\tau -1=0$. If $\nu$ is odd and $E(K)/E_{0}(K)
=0$ but 2 does not divide $\#\Gal (L'/K)$, then $E(L')/E_{0}(L')=0$ and
so there is no $P$. When $\nu$ is odd and $E(K)/E_{0}(K)=0$ and
$2$ divides $\#\Gal (L'/K)$,
we have $E(L')/E_{0}(L')\cong {\bf Z}/\nu {\bf Z}$.
Since
$\tau $ acts as $-1$, $\tau -1$ acts as multiplication by $-2$ and so
all points of $E(L')$ are in the image of $\tau -1$.

If $\nu$ is even and
$E(K)/E_{0}(K)\not\cong {\bf Z}/2{\bf Z}$ then
$E(K)/E_{0}(K)\cong {\bf Z}/\nu {\bf Z}
\cong E(M)/E_{0}(M)$ and so $\tau$ acts trivially on $E(M)/E_{0}(M)$.
If $\nu\equiv 2({\rm mod}\, 4)$,
$E(K)/E_{0}(K)\cong {\bf Z}/2{\bf Z}$
and $2$ does not divide $\#\Gal (L'/K)$,
then $E(L')/E_{0}(L')\cong {\bf Z}/2{\bf Z}$.
If 2 does not divide $\#\Gal (M/L')$ either then
$E(M)/E_{0}(M)\cong {\bf Z}/2{\bf Z}$ and $\tau$ acts trivially. Even if
$2$ divides $\#\Gal (M/L')$
then the image of $\tau -1$ will be points whose image in
$E(M)/E_{0}(M)
\cong {\bf Z}/\nu {\bf Z}$ are even. The point $P$,
being in $E(L')/E_{0}(L')\cong {\bf Z}/2{\bf Z}$,
will have image $\frac{1}{2}\nu$ in ${\bf Z}/\nu {\bf Z}$
which is odd so it is not in the image of $\tau -1$.
When
$2$ divides
$\#\Gal (L'/K)$, we have $E(L')/E_{0}(L')\cong {\bf Z}/\nu {\bf Z}$.
Since
$\tau$ acts as $-1$, the points whose images in
$E(L')/E_{0}(L')
\cong {\bf Z}/\nu {\bf Z}$ are even are in the image of $\tau -1$.
If $\nu\equiv 0({\rm mod}\, 4)$, $E(K)/E_{0}(K)\cong {\bf Z}/2{\bf Z}$, and
2 does not divide $\#\Gal (M/K)$,
then $E(M)/E_{0}(M)\cong {\bf Z}/2{\bf Z}$ and $\tau$ acts trivially. If
$2$ divides $\#\Gal (M/K)$, but 2 does not divide $\#\Gal (L'/K)$,
then $E(L')/E_{0}(L')\cong {\bf Z}/2{\bf Z}$
and so $2P$ has non-singular reduction.
Notice that since $\tau -1$ acts as multiplication by $-2$, that
$P$ is in the image of $\tau -1$. When
$2$ divides $\#\Gal (L'/K)$,
we have $E(L')/E_{0}(L')\cong {\bf Z}/\nu {\bf Z}$.
Thus the elements of $E(L')[\phi ]$ whose
images in
$E(M)/E_{0}(M)
\cong {\bf Z}/\nu {\bf Z}$ are even are in the image of $\tau -1$.

Let $E$ have type $II$ or $II^{\ast}$ reduction.
The group $E(K^{\rm unr})/
E_{0}(K^{\rm unr})$ is trivial so \linebreak
$E(M)/E_{0}(M)$ is also
and so there could not be a $P$. Let $E$ have type $III$ or $III^{\ast}$
reduction.
The group $E(K^{\rm unr})/
E_{0}(K^{\rm unr})$ is isomorphic to ${\bf Z}/2{\bf Z}$.
Since the
automorphism group of ${\bf Z}/2{\bf Z}$ is trivial, we see that
the action of $\tau$ is
trivial.

Let $E$ have type $IV$ or $IV^{\ast}$ reduction.
The group $E(K^{\rm unr})/
E_{0}(K^{\rm unr})$ is isomorphic to ${\bf Z}/3{\bf Z}$. The
automorphism group of
${\bf Z}/3{\bf Z}$ is isomorphic to ${\bf Z}/3{\bf Z}^{\ast}$.
If
$E(K)/E_{0}(K)\neq 0$
then $E(K)/E_{0}(K)\cong E(M)/E_{0}(M)$ and so
$\tau$ acts trivially.
If $E(K)/E_{0}(K)=0$
and 2 does not divide $\#\Gal (L'/K)$ then there is no $P$.
When
$E(K)/E_{0}(K)=0$ and $2$ divides $\#\Gal (L'/K)$ we have a $P$ defined
over $L'$ and $\tau -1$ acts as an automorphism on $E(M)/E_{0}(M)$
and so $P$ is in the image of $\tau -1$.

Let $E$ have type $I_{\nu}^{\ast}$ reduction. If $\nu$ is odd then
we have $E(K^{\rm unr})/
E_{0}(K^{\rm unr})\cong {\bf Z}/4{\bf Z}$. The automorphism group of
${\bf Z}/4{\bf Z}$ is isomorphic to
${\bf Z}/4{\bf Z}^{\ast}$. Therefore $E(K)/E_{0}(K)$
has 2 or 4 elements. If $E(K)/E_{0}(K)\cong {\bf Z}/4{\bf Z}$, then
$E(K)/E_{0}(K)\cong E(M)/E_{0}(M)$ and so
$\tau$ acts trivially. If $E(K)/E_{0}(K)\cong {\bf Z}/2{\bf Z}$, then
$P\in E(L')[\phi ]$ is in the image of $\tau -1$ if $2P$ has non-singular
reduction and $2$ divides $\#\Gal (M/K)$.

If $\nu$ is even, then
we have $E(K^{\rm unr})/
E_{0}(K^{\rm unr})\cong {\bf Z}/2{\bf Z}\times {\bf Z}/2{\bf Z}$.
When\linebreak
$E(K)/E_{0}(K)\cong {\bf Z}/2{\bf Z}\times {\bf Z}/2{\bf Z}$, we have
$E(K)/E_{0}(K)\cong E(M)/E_{0}(M)$ and so
$\tau$ acts trivially. When $E(K)/E_{0}(K)\cong {\bf Z}/2{\bf Z}$
the action of $\Gal (K^{\rm unr}/K)$ on $E(M)/E_{0}(M)$ has order 2.
Therefore, $P$ could only be in the image of $\tau -1$ if $\tau$ fixes
$P$ modulo $E_{0}(M)$ and $2$ divides $\#\Gal (M/K)$.
When $E(K)/E_{0}(K)=0$,
the action of $\Gal (K^{\rm unr}/K)$ on $E(M)/E_{0}(M)$ has order 3.
There will be no $P$ unless $3$ divides $\#\Gal (L'/K)$.
If that is the case, then $\tau -1$ is
an automorphism on $E(M)/E_{0}(M)$ and so all points $P$
of $E(L')$ will be in the
image of the map $\tau -1$.
\end{proof}

Using the algorithm of Tate, one can quickly determine the
reduction type of $E$ over $K$ and compute
the group $E(K)/E_{0}(K)$. As we will see in examples in
section~\ref{i2kpe}, it is easy to determine the action of $\tau$
on $E(M)/E_{0}(M)$.
So in practice, it is easy to determine
which elements of $E(M)[\phi ]$ are in the image of the map $\tau -1$
on $E(M)$. So for elliptic curves, we have an algorithm for computing
the size of the intersection quotient.

Let us now compute the order of the group $A'(K)/\phi A(K)$.
Let $\cal O$ be the ring of integers in $K$.
For $i\geq 1$ there is a natural homomorphism
\[
NA^{0}({\cal O})\rightarrow NA^{0}({\cal O}/\Gp^{i})\]
where we use $\Gp$ to denote the maximal ideal in $\cal O$.  The
kernel of this homomorphism is often denoted $NA^{i}({\cal O})$ and
it is isomorphic to a subgroup of the group $A_{0}(K)$ which we shall
denote by $A_{i}(K)$. If $p$ is the characteristic of the residue
class field of $K$, then the group $A_{1}(K)$ is a pro-$p$-group
which is isomorphic to the underlying group of a formal group (see
\cite{Mi}, p.\ 56).

There is some $m$ such that for all $n\geq m$,
the groups $A_{n}(K)$ and $A'_{n}(K)$ are isomorphic, as the
underlying groups of formal groups, to the product of $g$ copies of
the additive group of $\cal O$ where
$g$ is the dimension of $A$ and of $A'$ (see \cite{Ma}).  As $\phi$
is a $K$-defined isogeny of algebraic groups we can write $\phi$ as a
$g$-tuple of power series in $g$-variables in some neighborhoods
$A_{l}(K)$ and $A'_{l}(K)$ of the $0$-points for some $l\geq m$.  If
$k$ is in $K$, then let $|k|$ be its normalized absolute value, by
which if $\pi$ is a prime element of $K$ and $q$ is the order of the
residue class field we have $|\pi |= q^{-1}$.  Let $|\phi'(0)|$ be
the normalized absolute value of the determinant of the Jacobian
matrix of partials of $\phi$ evaluated at the 0-point.  The value of
$|\phi'(0)|$ does not depend on the choice of parameters.

\begin{lemma}
\label{size}
The order of the group $A'(K)/\phi A(K)$ is
\[
\frac{|\phi'(0)|^{-1}\cdot
\# A(K)[\phi ]
\cdot\# A'(K)/A'_{0}(K)}
{\# A(K)/A_{0}(K)}.
\]
\end{lemma}
\begin{proof}
From the snake lemma, we have the following commutative diagram
\[
\begin{array}{rcccccccl}
 & & 0 & \rightarrow & A(K)[\phi ] & \rightarrow & H_{1} & \rightarrow & \\
 & & \downarrow & & \downarrow & & \downarrow & & \\
 0 & \rightarrow & A_{l}(K) & \rightarrow & A(K) & \rightarrow &
 A(K)/A_{l}(K) & \rightarrow & 0 \\
 & & \;\;\;\downarrow \phi & & \;\;\;\downarrow \phi & &
 \;\;\;\downarrow \phi & & \\
 0 & \rightarrow & A'_{l}(K) & \rightarrow & A'(K) & \rightarrow &
 A'(K)/A'_{l}(K) & \rightarrow & 0 \\
 & & \downarrow & & \downarrow & & \downarrow & & \\
 & \rightarrow & A'_{l}(K)/\phi A_{l}(K) & \rightarrow &
 A'(K)/\phi A(K) & \rightarrow & H_{2} & \rightarrow & 0
\end{array}
\]
with $l$ as above.
So we have\[
\# A'(K)/\phi A(K)=\# A(K)[\phi]\cdot \# A'_{l}(K)/\phi A_{l}(K)\cdot
\frac{\# H_{2}}{\# H_{1}}\]\[ =
\# A(K)[\phi]\cdot \# A'_{l}(K)/\phi A_{l}(K)\cdot
\frac{\# A'(K)/A'_{l}(K)}{\# A(K)/A_{l}(K)}
\]

Let us first compute $\# A'_{l}(K)/\phi A_{l}(K)$.
Fix isomorphisms of $A_{l}(K)$
and $A'_{l}(K)$ with ${\cal O}^{g}$.
Let $\mu_{A}$ be the Haar
measure on $A_{m}(K)$ which is induced by the product
of the Haar measures on $\cal O$ which have the property that
$\mu({\cal O})=1$ and $\mu(\Gp^{i})=|\pi^{i}|$ for $i\geq 0$.
Define $\mu_{A'}$ similarly.
The isogeny $\phi$ extends to a morphism of the Neron models which respects
the maps from $NA({\cal O})$ to $NA({\cal O}/\Gp^{n})$
and also these maps for $A'$.
We have \[
\int\limits_{A'_{l}(K)}d\mu_{A'}=
\int\limits_{A_{l}(K)}d\mu_{A}=
\int\limits_{\phi A_{l}(K)}|\phi'(0)|^{-1}d\mu_{A'}=
|\phi'(0)|^{-1}\int\limits_{\phi A_{l}(K)}d\mu_{A'}.\]
If $H\subseteq G$ are subgroups of
$A'_{m}(K)$ then $[G:H]=\mu_{A'}(G)/\mu_{A'}(H)$.
Therefore \[
\# A'_{l}(K)/\phi A_{l}(K) = |\phi'(0)|^{-1}.\]

Now let us compute $\# A(K)/A_{l}(K)$ and also for $A'$. For all $i\geq 1$
the quantities $\# A_{i}(K)/A_{i+1}(K)$, $\# A'_{i}(K)/A'_{i+1}(K)$ and
$q^{g}$ are the same. Since $A$ and $A'$ are isogenous, we also have
$\# A_{0}(K)/A_{1}(K)$ equals $\# A'_{0}(K)/A'_{1}(K)$. Therefore the
quotient of
$\# A'(K)/A'_{l}(K)$ by $\# A(K)/A_{l}(K)$ equals the quotient of $\# A'(K)/
A'_{0}(K)$ by $\# A(K)/A_{0}(K)$.
\end{proof}

If $A$ and $A'$ are elliptic curves $E$ and $E'$, then
one can use Tate's algorithm to compute the orders of
$E'(K)/E'_{0}(K)$ and
$E(K)/E_{0}(K)$. The quantity $|\phi'(0)|^{-1}$ is simply the leading coeffient
of the power series representation of $\phi$.
We can compute that using
the following method. First
write $\phi$ explicitly as coordinate functions $\phi (x,y)=
(x',y')$ (see \cite{Ve}).
Then make the substitutions $z=-x/y$ and $z'=-x'/y'$ and start to
write
$z'$ as a power series in $z$
(see \cite{Si}, chapter IV).

For arbitrary abelian varieties,
if the residue characteristic $p$ does not divide the degree of the
isogeny, then $|\phi'(0)|$ is 1, since the neighborhoods of the
0-points are pro-$p$-groups.
This follows from the last equation in the above proof.
If $A$ has good reduction over $K$ then
$A(K)=A_{0}(K)$ and also $A'(K)=A'_{0}(K)$. If $A$ does not have good
reduction over $K$
then it is not always the case that the size of
$A(K)/A_{0}(K)$ is equal
to the size of $A'(K)/A'_{0}(K)$.
As an example, let $E$ be the elliptic curve defined by the equation
$Y^{2}=X^{3}-189X+1269$ over ${\bf Q}_{31}$
and let $\phi$ be the rational 3-isogeny
gotten by dividing out by the 0-point and
the points $(3,\pm 27)$.
The discriminant of $E$ is $\Delta_{E}=-2^{4}\cdot 3^{12}\cdot 31$. This curve
has split multiplicative reduction over ${\bf Q}_{31}$.
Recall that the
order of the group $E({\bf Q}_{31})/E_{0}({\bf Q}_{31})$ is the same as
the valuation of the discriminant for curves of split multiplicative
reduction.
So the order
of the
group $E({\bf Q}_{31})/E_{0}({\bf Q}_{31})$ is 1.
The curve $E'=E/E[\phi ]$ has Weierstrass equation
$Y^{2}=X^{3}+1431X-12339$. This curve has discriminant
$\Delta_{E'}=-2^{4}\cdot 3^{12}\cdot 31^{3}$
and also has split multiplicative reduction over ${\bf Q}_{31}$.
Thus the order of the group $E'({\bf Q}_{31})/E'_{0}({\bf Q}_{31})$ is 3.

The size of the group $A'(K)/\phi A(K)$ is much easier to compute if
$\phi$ is a multiplication by $n$ map because then we can take
$A=A'$.

\begin{prop}
\label{pr:sizen}
If $K$ is a finite extension of ${\bf Q}_{p}$ and $A$ has dimension $g$ and
if $r={\rm ord}_{p}(n)$ then
$\# A(K)/nA(K)=p^{gr[K:{\bf Q}_{p}]}\cdot \#A(K)[n]$.
\end{prop}

\begin{proof}
Let $A_{m}(K)$ be isomorphic to ${\cal O}^{g}$.
Then we have
\[
\begin{array}{rcccccccl}
 & & 0 & \rightarrow & A(K)[n] & \rightarrow & H_{1} & \rightarrow & \\
 & & \downarrow & & \downarrow & & \downarrow & & \\
 0 & \rightarrow & A_{m}(K) & \rightarrow & A(K) & \rightarrow &
 A(K)/A_{m}(K) & \rightarrow & 0 \\
 & & \;\;\;\;\;\;\downarrow [n] & & \;\;\;\;\;\;\downarrow [n] & &
\;\;\;\;\;\;\downarrow [n] & & \\
 0 & \rightarrow & A_{m}(K) & \rightarrow & A(K) & \rightarrow &
 A(K)/A_{m}(K) & \rightarrow & 0 \\
 & & \downarrow & & \downarrow & & \downarrow & & \\
 & \rightarrow & A_{m}(K)/nA_{m}(K) & \rightarrow &
 A(K)/nA(K) & \rightarrow & H_{2} & \rightarrow & 0.
\end{array}
\]
It is clear that $H_{1}$ and $H_{2}$ have the same size and that
$\# A_{m}(K)/nA_{m}(K)=p^{gr[K:{\bf Q}_{p}]}$.
\end{proof}

To close this section, let us discuss what happens for the completion of
a number field at an infinite prime.

\begin{lemma}
\label{le:arch}
Let $K$ be the completion of a number field at an infinite prime.
Let $\phi$ be a $K$-defined isogeny from $A$ onto $A'$, abelian varieties
which are defined over $K$. Let $L=K(A[\phi ])$
and $H^{1}(\Gal (L/K),A[\phi ])=0$. The group of unramified
homomorphisms in $\Hom_{\Gal (L/K)}(\Gal (\overline{L}/L),A[\phi])$ is
trivial.
If $K\cong {\bf C}$ then $A'(K)/\phi A(K)=0$.
If $K\cong {\bf R}$ and the degree of $\phi$ is odd, then $A'(K)/\phi A(K)=0$.
If $K\cong {\bf R}$ and there are points of $A[\phi ]$ that are not defined
over ${\bf R}$ then $A'(K)/\phi A(K)=0$.
Otherwise $A'(K)/\phi A(K)$ has exponent 2 and
we have
$\# A'(K)/\phi A(K)\leq 2^{g}$.
\end{lemma}
\begin{proof}
Since ${\bf R}$ and ${\bf C}$ have no non-trivial unramified extensions,
the only homomorphism from $\Gal (\overline{L}/L)$
to $A[\phi ]$ which is unramified
is the trivial homomorphism.
If $K\cong {\bf C}$ then
$\phi$ maps onto $A'(K)$ so $A'(K)/\phi A(K)=0$. Let $K\cong {\bf R}$.
If the degree of $\phi$ is odd, then $A'(K)/\phi A(K)$ embeds into
$H^{1}(\Gal ({\bf C}/{\bf R}),A[\phi ])$ which is trivial since $[{\bf C}:
{\bf R}]$
is coprime to the order of $A[\phi ]$ (see \cite{AW}, p.\ 105).
If there are points of $A[\phi ]$ that are not defined over ${\bf R}$ then
$L \cong {\bf C}$ and so $H^{1}(\Gal ({\bf C}/{\bf R}),A[\phi ])$ is trivial
by assumption.

The only other case is that $K\cong {\bf R}$, the degree of $\phi$ is even
and all of the points of $A[\phi ]$ are defined over ${\bf R}$.
We would like to show that $A'({\bf R})$ has at
most $2^{g}$ components.
Let ${\bf C}^{g}$ denote the tangent space to
$A'({\bf C})$ at the 0-point of $A'({\bf C})$. Since $A'$ is defined
over ${\bf R}$ there is a $2g$-dimensional lattice
$\Lambda$ fixed by the action of $\Gal ({\bf C}/{\bf R})$
with
$A'({\bf C})$ isomorphic to the quotient of ${\bf C}^{g}$
by $\Lambda$
(see \cite{Ro}). We can look at this as an exact sequence of $\Gal
({\bf C}/{\bf R})$-modules
\[
0\rightarrow \Lambda\rightarrow {\bf C}^{g}\rightarrow A'({\bf C})\rightarrow
0.\]
By taking $\Gal ({\bf C}/{\bf R})$-invariants, we get the following exact
sequence of groups, where ${\bf R}^{g}$ denotes the tangent space
to $A'({\bf R})$ at its 0-element, \[
0\rightarrow\Lambda \cap {\bf R}^{g}\rightarrow {\bf R}^{g}\rightarrow
A'({\bf R})\rightarrow H^{1}(\Gal ({\bf C}/{\bf R}),\Lambda )\rightarrow 0.\]
Since $H^{1}(\Gal ({\bf C}/{\bf R}),{\bf C}^{g})=0$ (see \cite{Se}, p.\ 152),
we get a 0 at the right end of that exact sequence.
Let $\sigma$ generate
$\Gal ({\bf C}/{\bf R})$. We have\[
H^{1}(\Gal ({\bf C}/{\bf R}),\Lambda )\cong \frac{{\rm ker}(\sigma +1):
\Lambda\rightarrow\Lambda}{(\sigma -1)\Lambda}\]
Now the kernel of $\sigma +1$ on ${\bf C}^{g}$ is a $g$-dimensional vector
space over ${\bf R}$ so the kernel of $\sigma +1$ on $\Lambda$ is a lattice
of dimension $g$. Since 2 annihilates the cohomology group
(see \cite{AW}, p.\ 105),
the cohomology
group has size at most $2^{g}$.
The image of ${\bf R}^{g}$ is connected,
so $A'({\bf R})$ has at most $2^{g}$ components.

The isogeny $\phi$ takes
components onto components and so the size of the group \linebreak
$A'({\bf R})/\phi
A({\bf R})$ is at most $2^{g}$. The group $A'({\bf R})/\phi A({\bf R})$
embeds into $H^{1}(\Gal ({\bf C}/{\bf R}),A[\phi ])$ which 2 annihilates,
so $A'({\bf R})/\phi
A({\bf R})$ has exponent 2.
\end{proof}

If $A$ and $A'$
are elliptic curves $E$ and $E'$, then we can compute the
size of $E'(K)/\phi E(K)$ exactly.
If $K\cong {\bf R}$
we can write down a Weierstrass equation of the form
$Y^{2}=f$ where $f\in {\bf R}[X]$.
The discriminant of the Weierstrass equation is 16 times the discriminant
of $f$. If $f$ has 3 real roots then these discriminants are positive.
If $f$ has 2 imaginary roots then these discriminants are negative.
When the discriminant is positive we can write down a Weierstrass
equation of the form
$Y^{2}=(X-a)(X-b)(X-c)$ where $a,b,c$ are real and $a<b<c$.
\begin{prop}
\label{pr:archec}
Let $K$ be the completion of a number
field at an infinite prime.
Let $E$ and $E'$ be elliptic curves defined over $K$ and let $\phi$
be a $K$-defined isogeny of $E$ onto $E'$.
Let $L=K(E[\phi ])$
and $H^{1}(\Gal (L/K),E[\phi ])=0$.
The group $E'(K)/\phi E(K)$ is trivial
unless $K\cong {\bf R}$, the degree of $\phi$ is even, all of
the points of $E[\phi ]$ are defined over $K$
and we are in one of the following cases when the group will have order 2.
\begin{itemize}
\item
The discriminant of $E$ is negative.
\item
The discriminant of $E$ is positive
and the 2-torsion point $(a,0)$ is contained in $E[\phi ]$.
\end{itemize}
\end{prop}
\begin{proof}
The only thing left to do after the proof of lemma~\ref{le:arch}
is show the size of the group $E'(K)/\phi E(K)$
in the case that $K\cong {\bf R}$, the degree of $\phi$ is even and all of
the points of $E[\phi ]$ are defined over $K$.
To ease notation we will say that $K={\bf R}$.

Since $E$ is defined over ${\bf R}$,
the group of points of $E$ over the complex numbers is isomorphic to
the complex numbers modulo a lattice generated by 1 and $\alpha i$
(when $\Delta_{E}>0$)
or $1/2+\alpha i$ (when $\Delta_{E}<0$) where
$\alpha$ is a positive real number (see \cite{Si}, chapter VI).

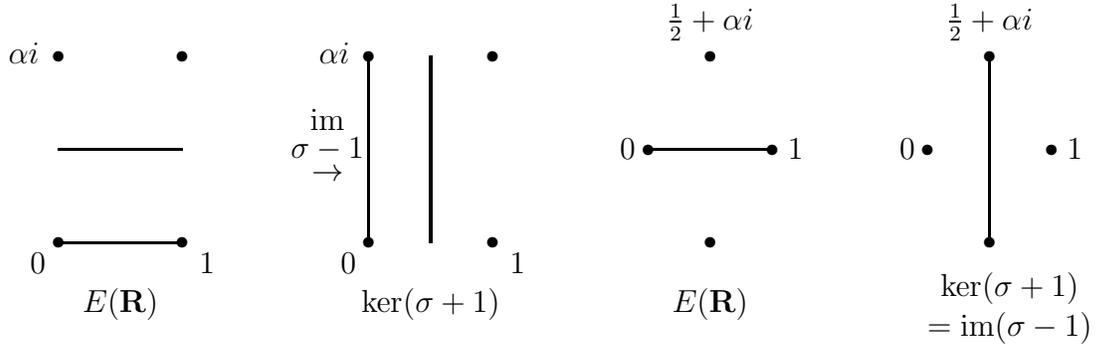
\begin{picture}(4000,1300)
\put(0,100){\begin{picture}(900,1200)
\put(150,300){\circle*{30}}
\put(150,900){\circle*{30}}
\put(550,300){\circle*{30}}
\put(550,900){\circle*{30}}
\thicklines
\put(150,300){\line(1,0){400}}
\put(150,600){\line(1,0){400}}
\put(350,100){\makebox(0,0){$E({\bf R})$}}
\put(60,200){0}
\put(605,200){1}
\put(40,910){\makebox(0,0){$\alpha i$}}\end{picture}}
\put(900,100){\begin{picture}(1000,1200)
\thicklines
\put(250,300){\circle*{30}}
\put(250,900){\circle*{30}}
\put(650,300){\circle*{30}}
\put(650,900){\circle*{30}}
\put(140,910){\makebox(0,0){$\alpha i$}}
\put(160,200){0}
\put(705,200){1}
\put(250,300){\line(0,1){600}}
\put(450,300){\line(0,1){600}}
\put(450,100){\makebox(0,0){${\rm ker}(\sigma +1)$}}
\put(0,500){\shortstack{im\\ $\sigma -1$ \\ $\rightarrow$}}
\end{picture}}	
\put(1800,100){\begin{picture}(900,1200)
\thicklines
\put(250,600){\circle*{30}}
\put(650,600){\circle*{30}}
\put(450,300){\circle*{30}}
\put(450,900){\circle*{30}}
\put(250,600){\line(1,0){400}}
\put(450,100){\makebox(0,0){$E({\bf R})$}}
\put(160,570){0}
\put(700,570){1}
\put(450,1010){\makebox(0,0){$\frac{1}{2}+\alpha i$}}\end{picture}}
\put(2700,100){\begin{picture}(1000,1200)
\thicklines
\put(250,600){\circle*{30}}
\put(650,600){\circle*{30}}
\put(450,300){\circle*{30}}
\put(450,900){\circle*{30}}
\put(160,570){0}
\put(700,570){1}
\put(450,300){\line(0,1){600}}
\put(250,0){\shortstack{${\rm ker}(\sigma +1)$\\ $={\rm im}(\sigma -1)$}}
\put(450,1010){\makebox(0,0){$\frac{1}{2}+\alpha i$}}\end{picture}}
\end{picture}

In the first case, $E({\bf R})$ is made
up of two components
with imaginary parts 0 and $\alpha /2$ in the given fundamental domain.
In the second case, $E({\bf R})$ is the subgroup of the given fundamental
domain with imaginary part 0.
We have the following exact
sequence
\[
0\rightarrow E'({\bf R})/\phi E({\bf R})\rightarrow
H^{1}({\bf R},E[\phi ]) \rightarrow
H^{1}({\bf R},E({\bf C}))[\phi ]\rightarrow 0.
\]

Let $\sigma$ generate $\Gal ({\bf C}/{\bf R})$ and
let the discriminant of $E$ be positive.
On $E({\bf C})$,
the kernel of the norm, $\sigma +1$, has two components,
the set of points with real part $1/2$ and the
set of points with real part 0.
The image of $\sigma -1$ is just the
set of points with real part 0. Thus the group $H^{1}({\bf R},
E({\bf C}))$ has order 2 and is generated by the component with real
part $1/2$.
To determine the order of $H^{1}({\bf R},
E({\bf C}))[\phi ]$, we just need to check whether or not an element of
the kernel of $\phi$ is on the component with real part $1/2$.
The points in the kernel of $\phi$ are all in $E({\bf R})$ and the points
of $E({\bf R})$
intersect the component with real part $1/2$ only at 2-torsion points.
Therefore if there is a 2-torsion point in $E[\phi ]$
with real part $1/2$, then all of $H^{1}({\bf R},E({\bf C}))$ is
$\phi$-torsion.
By graphing $E({\bf R})$ we see that
the 2-torsion point $(c,0)$ is on the component of $E({\bf R})$ that
includes the identity and so
maps to the point $1/2$ in the fundamental domain.
A quick computation with 2-isogenies (see \cite{Si}, p.\ 74)
shows that by dividing out $E$
by the subgroup generated by the 2-torsion point $(b,0)$,
that the quotient curve has negative
discriminant, and so the point $(b,0)$ must map to the point
$(1+\alpha i)/2$. Therefore the point $(a,0)$ must map to the point
$\alpha i/2$.
If all of $E[2]$ is contained in $E[\phi ]$ then the points $(b,0)$ and
$(c,0)$ are in $E[\phi ]$ so the group $H^{1}({\bf R},E({\bf C}))[\phi ]$
has order 2. The group $H^{1}({\bf R},E[\phi ])=\Hom
(\Gal ({\bf C}/{\bf R}),E[\phi ])$ will have order 4 so
the group $E'({\bf R})/\phi E({\bf R})$ will have order 2.
If $E[2]$ is not contained in $E[\phi ]$ then the group
$H^{1}({\bf R},E[\phi ])$
will have order 2. If the non-trivial 2-torsion
point in $E[\phi ]$ is $(a,0)$ then
the group $H^{1}({\bf R},E({\bf C}))[\phi ]$ will be trivial and have
order 2 otherwise. Therefore the group $E'({\bf R})/\phi E({\bf R})$ will have
order 2 if $(a,0)$ is the 2-torsion point in $E[\phi ]$ and will be
trivial otherwise.

If the discriminant is negative, then we see that the kernel of
$\sigma +1$ on $E({\bf C})$
is the same as the image of $\sigma -1$; it is the subgroup of the
fundamental domain with real part equal to $1/2$. Thus the
group $H^{1}({\bf R},E({\bf C}))[\phi ]$ is trivial.
Since the discrimimant of $E$ is negative, the group $E({\bf R})[2]$ has
order 2 so
the group $H^{1}({\bf R},
E[\phi ])$ has order 2. Therefore
$E'({\bf R})/\phi E({\bf R})$ has order 2.
\end{proof}

\section{Global computations}
\label{glob comp}

In this section
we will bring together the local results from the previous section
to prove a theorem giving upper bounds for the index of the
global intersection in the Selmer group and in the group related to an
ideal class group.
Let $A$ and $A'$ be abelian varieties of the same dimension $g$,
defined over $K$, an algebraic
number field, and let $\phi$ be a $K$-defined isogeny from
$A$ onto $A'$.
Let $L$ be the minimal field of definition of the points
in $A[\phi ]$ and assume that $H^{1}(G,A[\phi])$ is trivial
for all $G\subseteq\Gal (L/K)$.
Let $S^{\phi}(K,A)$ be the $\phi$-Selmer group of $A$ over $K$.
Let $C^{\phi}(K,A)$ be the subgroup of unramified homomorphisms in
$\Hom_{G(L/K)}(\Gal (\overline{L}/L),A[\phi ])$ and
let $I^{\phi}(K,A)$ be the intersection of
$C^{\phi}(K,A)$ and $S^{\phi}(K,A)$.

Let $\Gp$ be a prime of $K$.
Denote by $S^{\phi}(K_{\Gp},A)$ the group $A'(K_{\Gp})/\phi A(K_{\Gp})$.
Let $C^{\phi}(K_{\Gp},A)$ be the subgroup of unramified homomorphisms in
$\Hom_{G(LK_{\Gp}/K_{\Gp})}(\Gal (\overline{LK_{\Gp}}/LK_{\Gp}),A[\phi ])$
and let $I^{\phi}(K_{\Gp},A)$ be the intersection of $S^{\phi}(K_{\Gp},A)$
and $C^{\phi}(K_{\Gp},A)$.

Fix $\Gp$, a finite prime, and let $m_{\Gp}$ be the exponent of the
group $A(K_{\Gp})[\phi ]$. Denote by $M_{\Gp}$ the intersection of
the maximal unramified extension of $K_{\Gp}$ and the unramified
extension of $LK_{\Gp}$ of degree $m_{\Gp}$.
The extension $M_{\Gp}/K_{\Gp}$ is unramified and so the group
$\Gal (M_{\Gp}/K_{\Gp})$ is cyclic which we shall say is generated
by $\tau_{\Gp}$.

Let $NA$ and $NA'$ be the N\'{e}ron models of $A$ and $A'$ over $\cal
O$, the ring of integers in $K_{\Gp}$. Define $NA^{0}$ to be the open
subgroup scheme of $NA$ whose generic fiber is isomorphic to $A$ over
$K_{\Gp}$ and whose special fiber is the identity component of the
closed fiber of $NA$. The group $NA^{0}({\cal O})$ is isomorphic to a
subgroup of $A(K_{\Gp})$ which we shall denote by $A_{0}(K_{\Gp})$;
define $A'_{0}$ similarly.
If we compare at the 0-points of $A$ and $A'$, then $\phi$ can be written as
a $g$-tuple of power series in $g$ variables. Let $|\phi '(0)|$ be the
normalized
absolute value of the determinant of the Jacobian matrix of partials
for $\phi$ evaluated at the 0-point.

\begin{thm}
\label{1st thm}
Let $H^{1}(G,A[\phi ])=H^2(G,A[\phi])=0$ for all
$G\subseteq \Gal (L/K)$.
We have the following injections of groups
\[
S^{\phi}(K,A)/I^{\phi}(K,A)\hookrightarrow
\prod_{\Gp}S^{\phi}(K_{\Gp},A)/I^{\phi}(K_{\Gp},A)\]
\[
C^{\phi}(K,A)/I^{\phi}(K,A)\hookrightarrow
\prod_{\Gp}C^{\phi}(K_{\Gp},A)/I^{\phi}(K_{\Gp},A)\]
where $\Gp$ ranges over the primes of $K$.
For each finite prime $\Gp$ of $K$ that does not divide the conductor of
$A$ or the degree of $\phi$,
the groups $S^{\phi}(K_{\Gp},A)$, $I^{\phi}(K_{\Gp},A)$,
$C^{\phi}(K_{\Gp},A)$ and $A(K_{\Gp})[\phi ]$
are isomorphic.

The group $C^{\phi}(K_{\Gp},A)$ is trivial if $\Gp$ is infinite,
and isomorphic to $A(K_{\Gp})[\phi ]$ otherwise.
The group $I^{\phi}(K_{\Gp},A)$ is trivial if $\Gp$ is infinite, and
isomorphic to
\[
\frac{A(M_{\Gp})[\phi ]
\cap (\tau_{\Gp}-1)A(M_{\Gp})}{(\tau_{\Gp}-1)(A(M_{\Gp})[\phi ])}
\]
otherwise.
The group $S^{\phi}(K_{\Gp},A)$
is trivial if $\Gp$ is infinite
unless $\Gp$ is a real
prime, the degree of $\phi$ is even, and the points in $A[\phi ]$ are
defined over $K_{\Gp}$ in which case the group
has exponent 2 and has size at most
$2^{g}$.
The group
$S^{\phi}(K_{\Gp},A)$ has order
\[
\frac{|\phi '(0)|^{-1}
\cdot\# A(K_{\Gp})[\phi ]\cdot \# A'(K_{\Gp})/A'_{0}(K_{\Gp})}
{\# A(K_{\Gp})/A_{0}(K_{\Gp})}\]
when $\Gp$ is finite.
\end{thm}

\begin{proof}
The injections are clear from the definitions of the groups involved.
What happens at the infinite primes follows from lemma~\ref{le:arch}.
Let $\Gp$ be a finite prime. From
lemma~\ref{le:unr hom} we see that the group $C^{\phi}(K_{\Gp},A)$
is isomorphic to $A(K_{\Gp})[\phi ]$. Since $H^{1}(G,A[\phi])=0$ for all
$G\subseteq \Gal (L/K)$,
the group $I^{\phi}(K_{\Gp},A)$ is defined and
is isomorphic to the intersection quotient
from
theorem~\ref{thm:mainthm}. The order of the group $S^{\phi}(K_{\Gp},A)$
comes from lemma~\ref{size}.

Let $\Gp$ be a finite prime that does not divide the conductor of $A$
or the degree of $\phi$.
Since $\Gp$ does not divide the conductor
we have $A(M_{\Gp})=A_{0}(M_{\Gp})$.
So
from lemma~\ref{le:EmodE0}, all of the elements of
$A(M_{\Gp})[\phi ]$ are in $(\tau_{\Gp}-1)(A(M_{\Gp}))$.  Thus the
group $I^{\phi}(K_{\Gp},A)$ has the same order as $A(K_{\Gp})[\phi ]$
and so is equal to $C^{\phi}(K_{\Gp},A)$.  We also have $A(K_{\Gp})=
A_{0}(K_{\Gp})$
and $A'(K_{\Gp})=A_{0}(K_{\Gp})$. In addition, since $\Gp$ does not divide the
degree of $\phi$, we have $|\phi'(0)|=1$.  Therefore, the order of
$S^{\phi}(K_{\Gp},A)$ will be the same as the order of
$A(K_{\Gp})[\phi ]$. So the image of $S^{\phi}(K_{\Gp},A)$ in
$\Hom_{G(LK_{\Gp}/K_{\Gp})}(\Gal
(\overline{LK_{\Gp}}/LK_{\Gp}),A[\phi ])$ is the same as
$C^{\phi}(K_{\Gp},A)$, which is isomorphic to $A(K_{\Gp})[\phi ]$.
\end{proof}

For elliptic curves,
all of the local groups are easy to compute.
Use theorem~\ref{thm:mainthm},
lemma~\ref{le:EmodE0} and theorem~\ref{thm:cases}
to compute the order of the intersection quotient.
Theorem~\ref{thm:cases} requires a minimal Weierstrass equation at
the prime $\Gp$ which we may not have, but that does not matter. We
can just replace our Weierstrass equation with a minimal one at $\Gp$
so as to be able to apply
theorem~\ref{thm:cases} since the order of the intersection does not
depend on the embedding.
Use Tate's algorithm to compute the orders of
$E'(K_{\Gp})/E'_{0}(K_{\Gp})$ and $E(K_{\Gp})/E_{0}(K_{\Gp})$.
As noted after proving lemma~\ref{size}, the quantity
$|\phi '(0)|^{-1}$ can be computed easily using \cite{Ve} and \cite{Si},
chapter IV.
In the infinite case use proposition~\ref{pr:archec} to compute the size of
$E'(K_{\Gp})/\phi E(K_{\Gp})$.
The group $E(K_{\Gp})[\phi ]$ is trivial to compute also.

\section{The 2-map}
\label{2-map}

The 2-map has been the most studied because it is usually the easiest isogeny
to work with that is always defined over the field of definition of
an abelian variety. In addition, one does not have to work with two different
abelian varieties.
For these reasons,
2-Selmer groups are frequently used when trying to compute
the free ${\bf Z}$-rank of a Mordell-Weil group.
Birch and Swinnerton-Dyer \cite{BSD}
created an algorithm for computing 2-Selmer groups for elliptic curves
and this has been implemented as a computer program by John Cremona.
Recently, algorithms have been introduced for computing 2-Selmer groups
for the Jacobians of hyperelliptic curves that work in special cases.
For hyperelliptic curves of genus 2 there are the algorithms of Flynn
\cite{Fl}, and of Gordon and Grant \cite{GG},
and for arbitrary genus
see \cite{Sc}. The algorithms of Flynn and of Gordon and Grant have also been
implemented as computer programs.
As mentioned in the introduction,
the 2-map has
been studied in connection with its relation to the 2-rank
of the class group of cubic extensions of ${\bf Q}$.
We will see why in subsection~\ref{1stsub}

In this section we first show how the group $C^{2}(K,A)$ is related to
the 2-parts of class groups when $A$ is the Jacobian of an elliptic
or a hyperelliptic
curve.
We then present some examples of computing
the local intersection $I^{2}(K_{\Gp},E)$
for elliptic curves. Then we come
up with simpler upper bounds for the sizes of $S^{2}(K,A)/I^{2}(K,A)$
and $C^{2}(K,A)/I^{2}(K,A)$.
Lastly we do explicit computations for three examples of Jacobians of
elliptic and hyperelliptic curves.

\subsection {The group $C^{2}(K,A)$ for the Jacobians of elliptic and
hyperelliptic curves}
\label{1stsub}
Let $C$ be the projective
elliptic or hyperelliptic curve defined over $K$, a
number field, by the equation $Y^{2}=f$ where $f$ is a separable
polynomial of odd degree $d\geq 3$.
Let $J$ be the Jacobian of the normalization
of $C$ and let $\infty$ denote the point at infinity of the normalization of
$C$. The genus of the curve $C$ and the dimension of $J$ are both $(d-1)/2$.
Let $L=K(J[2])$, and let $G_{2}$ be a 2-Sylow subgroup of $\Gal
(L/K)$.  We will see below that when
$H^{1}(G_{2},J[2])$ is trivial that we have
\[
H^{1}(K, J[2])\cong  \Hom\nolimits_{\Gal (L/K)}(\Gal (\overline{L}/L),J[2]).\]
By abuse of notation we will denote the preimage in $H^{1}(K,J[2])$ of
$C^{2}(K,J)$ also by $C^{2}(K,J)$.

Let $F$ be the algebra defined by $F=K[T]/f(T)$ and let $\overline{F}$
be $\overline{K}[T]/f(T)$. In \cite{Sc} the author presented an isomorphism of
$H^{1}(K,J[2])$ and the kernel of the norm from $F^{\ast}/F^{\ast 2}$
to $K^{\ast}/K^{\ast 2}$.
Any subset with $d-1$ elements
of the set of $(\alpha_{i},0)-\infty$, where $f(\alpha_{i})=0$,
forms a basis for $J[2]$.
Let $w$ be the Weil-pairing of an element of
$J[2]$ with all $d$ of the points $(\alpha_{i},0)-\infty$;
then $w$ is a map from $J[2]$ to $\mu_{2}(\overline{F})$.
The map $w$ induces an injective map from $H^{1}(K,J[2])$ to $H^{1}(K,
\mu_{2}(\overline{F}))$. The latter cohomology group is isomorphic to
$F^{\ast}/F^{\ast 2}$ by a Kummer map (see \cite{Se}, p.\ 152).
Composing the Kummer isomorphism
and the Weil-pairing induces the desired isomorphism.

Let us define
unramified elements in the kernel of the norm from $F^{\ast}/F^{\ast 2}$ to
$K^{\ast}/K^{\ast 2}$.
We can write $F\cong F_{1}\times\ldots\times F_{r}$ where each
$F_{i}$ is a field and we have
\[
F^{\ast}/F^{\ast 2}\cong F_{1}^{\ast}/F_{1}^{\ast 2}\times\ldots\times
F_{r}^{\ast}/F_{r}^{\ast 2}.\]
For an arbitrary number field $W$, let us define the unramified elements
of $W^{\ast}/W^{\ast 2}$ to be those which have the property that if
one adjoins the square root of a representative to $W$, one gets
a totally unramified extension of $W$.
We will say an element of $F^{\ast}/F^{\ast 2}$ is unramified if its
image in each $F_{i}^{\ast}/F_{i}^{\ast 2}$ is unramified.
The group $C^{2}(K,J)$ injects into
the kernel of the
norm from $F^{\ast}/F^{\ast 2}$ to $K^{\ast}/K^{\ast 2}$. We will show
that its image is the same as the subgroup of unramified elements.

\begin{lemma}
\label{tfae}
Let $f$ be a separable polynomial of odd degree $d\geq 3$, defined over $K$
a field of characteristic 0. Let
$C$ be the curve defined by $Y^{2}=f$,
and let $J$ be the Jacobian of the normalization of $C$. Let $L=K(J[2])$ and
$G_{2}$ be a 2-Sylow subgroup of $\Gal (L/K)$.
Then the following are
equivalent.
\begin{enumerate}
\item The group $H^{1}(G_{2},J[2])$ is trivial.
\item The group $J[2]$ is cohomologically trivial as a $G_{2}$-module.
\item The group $J[2]$ is cohomologically trivial as a $\Gal (L/K)$-module.
\item Let $W$ be the field
fixed by $G_{2}$. In $W[X]$, the polynomial
$f$ is the product of one linear factor and $(d-1)/\# G_{2}$ irreducible
factors, each of degree $\# G_{2}$.
\item Let $\alpha_{i}$ for
$1\leq i \leq d$ be the roots of $f$.
For $i$ not equal to $j$,
the degree of $L$ over $K(\alpha_{i},\alpha_{j})$ is odd.
\end{enumerate}
\end{lemma}
\begin{proof}
From \cite{AW}, theorem 6, $(1)$ is equivalent to $(2)$ and from
\cite{AW}, proposition 8, $(2)$
is equivalent to
$(3)$. To prove that $(4)$ implies $(5)$ we note that $(4)$
implies that if $\sigma\in G_{2}$ fixes two different $\alpha_{i}$'s
then $\sigma$ is the identity. Let $i$ not equal $j$ and let
$\hat{G}=\Gal (L/K(\alpha_{i},\alpha_{j}))$. Let $\hat{H}$ be a
2-Sylow subgroup in $\hat{G}$. The group $\hat{H}$ can be extended
to a 2-Sylow subgroup of $\Gal (L/K)$, which we will call $G_{2}$.
Elements of $\hat{H}$ are in $G_{2}$ and fix two different $\alpha_{i}$'s
so $\hat{H}$ is trivial. Thus the degree of $L$ over $K(\alpha_{i},
\alpha_{j})$ is odd.
To prove that $(5)$ implies $(4)$ we note that $(5)$ implies that
$L$ is of odd degree over $W(\alpha_{i},\alpha_{j})$ when $i$ is not equal
to $j$. So $L$ equals $W(\alpha_{i},\alpha_{j})$. Since the degree of $f$
is odd, one of its roots must be in $W$ and each of the rest
generates $L$, which
is $(4)$.

Now let us prove that $(1)$ and $(4)$ are equivalent.
Let $X$ be the $G_{2}$-set
of roots
of $f$. Let $G_{2}$ act trivally on ${\bf F}_{2}$ and let ${\bf F}_{2}^{X}$
be the $G_{2}$-maps from $X$ to
${\bf F}_{2}$.
Recall that
any subset of the set of $(\alpha_{i},0)-\infty$
with $d-1$ elements forms a basis for $J[2]$.
It is easy to check that
the kernel of the norm from ${\bf F}_{2}^{X}$
to ${\bf F}_{2}$ is isomorphic as a $G_{2}$-module to $J[2]$. In addition,
since $d$ is odd, the composition of
the diagonal embedding of ${\bf F}_{2}$ in ${\bf F}_{2}^{X}$ with the
norm is the identity map so the exact sequence \[
0\rightarrow J[2]\rightarrow {\bf F}_{2}^{X}\stackrel{\rm norm}{\rightarrow}
{\bf F}_{2}\rightarrow 0\] splits.
Therefore the exact sequence \[
0\rightarrow H^{1}(G_{2},J[2])\rightarrow H^{1}(G_{2},{\bf F}_{2}^{X})
\rightarrow H^{1}(G_{2},{\bf F}_{2})\rightarrow 0\] also splits.
Because these groups are finite,
condition $(1)$ is equivalent to $H^{1}(G_{2},{\bf F}_{2}^{X})\cong
H^{1}(G_{2},
{\bf F}_{2})$.
Now $X$ is isomorphic as a $G_{2}$-set to the disjoint union of its orbits
and so to $\amalg
G_{2}/H_{j}$ for some set of subgroups $H_{j}$ of
$G_{2}$. Thus \[
{\bf F}_{2}^{X}\cong \bigoplus\limits_{j} {\bf F}_{2}^{G_{2}/H_{j}}\]
as $G_{2}$-modules. So we have
\[
H^{1}(G_{2},{\bf F}_{2}^{X})\cong H^{1}(G_{2},{\bf F}_{2}^{G_{2}/H_{1}})\oplus
\ldots \oplus H^{1}(G_{2},{\bf F}_{2}^{G_{2}/H_{n}})\] for some $n$.
Now from Shapiro's lemma (see \cite{AW}, proposition 2),
condition $(1)$ is equivalent to
\[
H^{1}(G_{2},{\bf F}_{2})\cong
H^{1}(H_{1},{\bf F}_{2})\oplus\ldots\oplus H^{1}(H_{n},{\bf F}_{2}).\]

Now $X$ has odd cardinality and $G_{2}$ is a 2-Sylow subgroup, so there
is an orbit of cardinality one. Without loss of generality we will
let $H_{1}=G_{2}$. So $(1)$ is equivalent to all of the groups
$H^{1}(H_{j},{\bf F}_{2})$ being trivial for
$j>1$.
Since $H_{j}$ acts trivially on ${\bf F}_{2}$ we know that
such cohomology groups are trivial only when $H_{j}$ is trivial.
So $(1)$ is true if and only if
$X$ has one orbit with a single element and otherwise consists of
orbits with the same cardinality as $G_{2}$.
In other words, $f$ is the product of one
linear factor and $(d-1)/\# G_{2}$ irreducible factors each of
degree $\# G_{2}$, which is condition $(4)$.
\end{proof}

Condition $(5)$ is always true for elliptic curves, and for curves of the
form $y^{2}=x^{p}+a$ where $p$ is an odd prime.

\begin{lemma}
\label{unr el}
Let $W$ be a finite extension of $K$, fields of characteristic 0.
The norm from $W^{\ast}/W^{\ast 2}$ to $K^{\ast}/K^{\ast 2}$ sends
unramified elements to unramified elements.
\end{lemma}
\begin{proof}
Let $W_{1}$ be any finite extension of $K$, and let
$W_{i}$, for $i=1$ to $s$, be the $s$
conjugates of $W_{1}$ over $K$.
Let $\alpha_{1}\in W_{1}$ and
let $\alpha_{i}$ be the image of $\alpha_{1}$
in each $W_{i}$. Let $W_{1}(\sqrt{\alpha_{1}})$ be a totally
unramified extension
of $W_{1}$. We will abbreviate totally unramified with unramified.
Then $W_{i}(\sqrt{\alpha_{i}})$ is an unramified extension of
$W_{i}$ for all $i$. Let $E$ be the normal closure of
$W_{1}(\sqrt{\alpha_{1}})$ over $K$. We want to prove that
$K(\sqrt{N_{W_{1}/K}(\alpha_{1})})$
is an unramified extension of
$K$.
Let $I\subseteq\Gal (E/K)$ be an inertia group and pick
$\sigma\in I$.
We want to prove that $\sigma$ fixes
$\sqrt{N_{W_{1}/K}(\alpha_{1})}=
\sqrt{\alpha_{1}\cdot\ldots
\cdot\alpha_{s}}$. Pick an orbit of the set of $\alpha_{i}$'s under
the action of the group generated by $\sigma$ and number them
$\alpha_{1},\ldots, \alpha_{j}$ where $\sigma(\alpha_{i})=\alpha_{i+1}$
with the subscripts computed modulo $j$.
We have $\sigma(\sqrt{\alpha_{i}})=
\pm \sqrt{\alpha_{i+1}}$ and if we compute these for $i$ from 1 to $j$
then the
number of minus signs must be even. If it were odd, then $\sigma^{j}
(\sqrt{\alpha_{1}})=-\sqrt{\alpha_{1}}$ and so $\sigma^{j}(\alpha_{1})=
\alpha_{1}$ and so $\sigma^{j}$ would be in $\Gal (E/W_{i})$.
Since $W_{i}(\sqrt{\alpha_{i}})$ is an unramified extension of $W_{i}$ we
know that $\sigma^{j}\in I$ would also fix $\sqrt{\alpha_{i}}$,
a contradiction. Thus there are an even number of minus signs and
so $\sigma$ fixes the square root of the product of the elements of each
orbit and so
$\sigma$ fixes $\sqrt{\alpha_{1}\cdot\ldots\cdot\alpha_{s}}$.
\end{proof}
If the element $a$ of $F$ has
image $(a_{1},\ldots ,a_{r})$ in $\prod F_{i}$, then $N_{F/K}(a)$ is the same
as the product of the $N_{F_{i}/K}(a_{i})$.
We see that the norm from $F^{\ast}/F^{\ast 2}
=\prod F_{i}^{\ast}/F_{i}^{\ast 2}$ to $K^{\ast}/K^{\ast 2}$
takes unramified elements to unramified elements.

\begin{thm}
\label{kernorm}
Let $f$ be a separable polynomial of odd degree $d\geq 3$, defined over $K$,
a number field. Let
$C$ be the curve defined by $Y^{2}=f$,
let $J$ be the Jacobian of the normalization of $C$ and let
$F=K[T]/f(T)$. Let $L=K(J[2])$ and
$G_{2}$ be a 2-Sylow subgroup of $\Gal (L/K)$.
If the group $H^{1}(G_{2},J[2])$ is
trivial then
the group $C^{2}(K,J)$ is isomorphic to
the kernel of the norm from the unramified elements in $F^{\ast}/F^{\ast 2}$
to the unramified elements in $K^{\ast}/K^{\ast 2}$.
\end{thm}
\begin{proof}
From lemma~\ref{tfae}, since $H^{1}(G_{2},J[2])$ is trivial, we know $J[2]$ is
cohomologically trivial as a $\Gal (L/K)$-module. So
$H^{i}(\Gal (L/K),
J[2])$ is trivial for all $i$.
From the extended inflation-restriction sequence (see \cite{Kc}, p.\ 30)
we then have\[
H^{1}(K, J[2])\cong  \Hom\nolimits_{\Gal (L/K)}(\Gal (\overline{L}/L),J[2]).\]
By abuse of notation we will denote the preimage in $H^{1}(K,J[2])$ of
$C^{2}(K,J)$ also by $C^{2}(K,J)$. Thus, $C^{2}(K,J)$ is the subgroup
of $H^{1}(K,J[2])$ consisting of classes of cocycles, which when restricted
to $\Gal (\overline{L}/L)$ factor through $\Gal (H(L)/L)$ where $H(L)$
is the Hilbert class field of $L$.

We have seen that the group $H^{1}(K,J[2])$ is isomorphic to the
kernel of the norm from $F^{\ast}/F^{\ast 2}$ to $K^{\ast}/K^{\ast
2}$.  An element $a$ of the kernel of the norm from $F^{\ast}/F^{\ast
2}$ to $K^{\ast}/K^{\ast 2}$ is in the image of $C^{2}(K,J)$ exactly
when its image $(a_{1},\ldots ,a_{r})$ in $\prod
F_{i}^{\ast}/F_{i}^{\ast 2}$ has the property that for each $i$, the
field $L(\sqrt{a_{i}})$ is unramified over $L$ with the $F_{i}$'s
properly embedded in $L$.
That is, for each
$K$-algebra homomorphism $\sigma$ from $F$ to $L$, the field
$L(\sqrt{\sigma(a)})$ is unramified over $L$. An element $a$ in the
kernel of the norm from $F^{\ast}/F^{\ast 2}$ to $K^{\ast}/K^{\ast
2}$ is in the subgroup of unramified elements exactly when for each
$K$-algebra homomorphism $\sigma$ from $F$ to $L$, the field
$\sigma(F)(\sqrt{\sigma(a)})$ is unramified over $\sigma(F)$.
We are trying to show that these two subgroups of $F^{\ast}/F^{\ast 2}$
are the same.
It is
clear that the latter subgroup is contained in the former.

Let us prove that the former subgroup is contained in the latter.
Let $a$ be an element of $F^{\ast}$ with $N_{F/K}(a)$ in $K^{\ast 2}$ and
with the property that
for every $K$-algebra homomorphism $\sigma$ from $F$ to $L$ that
$L(\sqrt{\sigma(a)})$ is unramified over $L$.
We have assumed condition $(1)$ of lemma~\ref{tfae}
and so we have condition $(4)$.
First we will assume that $K=W$, where $W$ is
the fixed field of $G_{2}$, a 2-Sylow
subgroup of $\Gal (L/K)$. Under this assumption, the algebra $F$ is
isomorphic to the direct product of $K$ and $(d-1)/\#\Gal (L/K)$ copies
of $L$. Let the image of $a$ in this direct product
be $(a_{1},\ldots, a_{r})$.
We need to show that $K(\sqrt{a_{1}})$ is an unramified extension of $K$
and that $L(\sqrt{a_{i}})$ is an unramified extension of $L$
for each $i>1$. The latter statement is part of the assumption.
Now $N_{F/K}(a)$ is equal to $a_{1}N_{L/K}(a_{2})\cdot\ldots\cdot
N_{L/K}(a_{r})$ and is in $K^{\ast 2}$. Thus $a_{1}$ is congruent
to $N_{L/K}(a_{2}\cdot\ldots\cdot a_{r})$ modulo $K^{\ast 2}$. By
hypothesis, $L(\sqrt{a_{i}})$ for $i>1$ is an unramified extension of $L$.
By lemma~\ref{unr el},
the field $K(\sqrt{a_{1}})$ must be an unramified extension
of $K$. So for every $K$-algebra homomorphism $\sigma$ from $F$ to $L$
we have $\sigma(F)(\sqrt{\sigma(a)})$ is unramified over $\sigma(F)$.

Now we will no longer assume that $K=W$ and reduce to that case. Assume
$a$ is in $F^{\ast}$ with $N_{F/K}(a)$ in $K^{\ast 2}$ and assume
for every $K$-algebra homomorphism $\sigma$ from $F$ to $L$ that
$L(\sqrt{\sigma(a)})$ is unramified over $L$.
Let $F\otimes_{K}W$ be isomorphic to the product of fields $\prod E_{i}$.
We have $\sum [E_{i}:F]=[W:K]$ which is odd so without loss of generality,
let the degree of $E_{1}$ over $F$ be odd. There is some homomorphism $\gamma$
from $W$ to $L$ such that the homomorphism $\sigma\otimes\gamma$ from
$F\otimes_{K} W$
to $L$ factors through $E_{1}$ and we can denote the induced map from
$E_{1}$ to its image by $\rho$. Now the map $\rho$
from $E_{1}$ to $\rho (E_{1})$ extends the map $\sigma$ from $F$ to
$\sigma (F)$.
\[
\begin{array}{ccccc}
\sigma\otimes\gamma: F\otimes_{K}W\;\;\rightarrow & E_{1} &
\stackrel{\rho}{\rightarrow} & \rho (E_{1}) &
\subset\;\; L \\
 & \cup & & \cup & \\
 & F & \stackrel{\sigma}{\rightarrow} & \sigma (F) & \end{array}\]
The degree of $E_{1}$ over $F$ is odd, so the degree
of $\rho (E_{1})$ over $\sigma (F)$ is odd. From above,
since $\sigma\otimes\gamma$ is a homomorphism from $F\otimes_{K}W$ to $L$,
we know $\rho (E_{1})(\sqrt{\rho(a\otimes 1)})$
over $\rho (E_{1})$ is unramified.
The only elements of inertia groups we have to worry about are of 2-power
order.
Since the degree of $\rho (E_{1})$ over
$\sigma (F)$ is odd and $\rho$ extends $\sigma$, we know that
$\sigma (F)(\sqrt{\sigma (a)})$ is unramified over $\sigma (F)$ also.
Therefore the subgroup of the kernel of the norm from $F^{\ast}/F^{\ast 2}$
to $K^{\ast}/K^{\ast 2}$ of unramified elements is equal to the image
of $C^{2}(K,J)$.
\end{proof}

The group of unramified elements in $F^{\ast}/F^{\ast 2}$ is
the product of the groups of unramified elements in each
$F_{i}^{\ast}/F_{i}^{\ast 2}$.
If $W$ is any field, such as $K$ or an $F_{i}$, then the
group of unramified
elements in $W^{\ast}/W^{\ast 2}$
is isomorphic to the dual of ${\rm Cl}(W)/{\rm Cl}(W)^{2}$ from
Kummer theory and class field theory.

\subsection{The group $I^{2}(K_{\Gp},E)$}
\label{i2kpe}

In this section we
will see that for elliptic curves, the computation of $I^{2}(K_{\Gp},E)$
is essentially a matter of finding the action of $\Gal (\overline{K_{\Gp}}/
K_{\Gp})$ on $E[2]$ and
using Tate's algorithm over $K_{\Gp}$
to deduce the action
of $\Gal (M_{\Gp}/K_{\Gp})$ on $E(M_{\Gp})/E_{0}(M_{\Gp})$.
This is facilitated by that fact that the possible
groups $E(M_{\Gp})/E_{0}(M_{\Gp})$ and the possible Galois actions on
them are very limited.

Let $E$ be an elliptic curve defined over
a number field $K$.
Condition $(5)$ of lemma~\ref{tfae} is satisfied for all elliptic curves
so condition $(3)$ holds. That says that
$H^{1}(G,E[2])=0$ for all groups
$G$ contained in $S_{3}$, where $S_{3}$ is the automorphism group of $E[2]$.
So the conditions of theorems~\ref{1st thm} and~\ref{thm:mainthm} hold.
{}From theorem~\ref{thm:mainthm} and lemmas~\ref{le:unr hom}
and~\ref{le:EmodE0},
the order of the local intersection, $I^{2}(K_{\Gp},E)$, is
the same as the order of the local group of unramified homomorphisms,
$C^{2}(K_{\Gp},E)$, if all of
the points in $E(M_{\Gp})[2]$ have non-singular reduction, which will happen
if the elliptic curve has good reduction at $\Gp$. So the only cases where
something interesting can happen occur when
$E$ has bad reduction at $\Gp$.

{}From theorem~\ref{thm:mainthm}, in order to compute the size of the local
intersection, we need to find out how many elements of $E(M_{\Gp})[2]$
are in $(\tau-1)E(M_{\Gp})$ and divide that quantity by the size of
$(\tau-1)(E(M_{\Gp})[2])$.
We know from lemma~\ref{le:EmodE0} that the elements of $E(M_{\Gp})[2]$
that are in $(\tau-1)E(M_{\Gp})$ are those
with
non-singular reduction and those with
singular reduction which
have the property that their image in $E(M_{\Gp})/E_{0}
(M_{\Gp})$ is in $(\tau-1)(E(M_{\Gp})/E_{0}(M_{\Gp}))$.
{}From theorem~\ref{thm:cases}
we find that the only times that a
2-torsion point with singular reduction could possibly be in the image of
$\tau -1$
are in the cases that
$E$ has type $I_{\nu}$ reduction where $4\!\mid\!\nu$
or type $I_{\nu}^{*}$
reduction with $\nu$ odd or even.
As we will see, in each of these three cases, a 2-torsion point with
singular reduction, may or may not be in the group
$(\tau-1)E(M_{\Gp})$.
In the other cases, either $E(K_{\Gp}^{\rm unr})/E_{0}(K_{\Gp}^{\rm unr})$
has no
2-torsion or a point with singular reduction generates the 2-part of
$E/E_{0}$ and can not be in the image of $\tau-1$.
Computations of the groups $E({\bf Q}_{p})/E_{0}({\bf Q}_{p})$
were accomplished
with Tate's algorithm \cite{Ta}.
\begin{enumerate}

\item {\it Examples of curves with type $I_{\nu}$ reduction with $4\!\mid\nu$
and 2-torsion points of singular reduction in and
not in the image of $\tau-1$}.
The curves $Y^{2}=X^{3}-26X^{2}+135X-567$ and
$Y^{2}=X^{3}+26X^{2}+135X+567$
have discriminant
$\Delta =-2^{4}\cdot 3^{4}\cdot 23^{2}\cdot 239$.
Over ${\bf Q}_{3}$ they each have type
$I_{\nu}$ reduction (multiplicative) with $\nu =4$. Each have two 2-torsion
points with singular reduction defined over ${\bf Q}_{3}$.
They differ by a 2-torsion point with non-singular reduction and so have the
same image in $E/E_{0}$.
For both curves, the group
$E({\bf Q}_{3}^{\rm unr})/E_{0}({\bf Q}_{3}^{\rm unr})$ is isomorphic to
${\bf Z}/4{\bf Z}$.
The first curve has
split multiplicative reduction over ${\bf Q}_{3}$
and the second does not. Therefore
the group $E({\bf Q}_{3})/E_{0}({\bf Q}_{3})$ is isomorphic to
${\bf Z}/4{\bf Z}$ for the first curve and ${\bf Z}/2{\bf Z}$ for the second.
The image of the 2-torsion points with singular reduction is the element
of order 2 in both groups.
All of the 2-torsion is defined over ${\bf Q}_{3}$ for these curves, so
$M_{3}$
is the quadratic unramified extension of ${\bf Q}_{3}$.
We have $E(M_{3})/E_{0}(M_{3})$ isomorphic to ${\bf Z}/4{\bf Z}$ for both
curves. The
automorphism $\tau$ acts trivially on each element of $E(M_{3})/E_{0}(M_{3})$
for the first curve and as $-1$ for the second.
The 2-torsion points with singular reduction
are in the image of the map
$\tau -1$ for the second curve, but not the first.

For both curves the size of $E(M_{3})[2]$ is 4 and $\tau$ acts trivially on
each element so $(\tau-1)(E(M_{3})[2])$ is trivial.
For both curves the size of $E_{0}(M_{3})[2]$ is 2
and from lemma~\ref{le:EmodE0}, those 2 points are automatically in
$(\tau-1)(E(M_{3}))$.
From above, the two 2-torsion points with singular reduction on the first
curve are not in the image of $\tau -1$.
Thus the size of $(\tau-1)E(M_{3})\cap E(M_{3})[2]$ is 2 and so the size of
$I^{2}({\bf Q}_{3},E)$ is 2.
The two 2-torsion points with singular reduction on the second curve are in
the image of $\tau -1$.
So $(\tau-1)E(M_{3})\cap E(M_{3})[2]$ and $I^{2}({\bf Q}_{3},E)$
each have size 4.

\item {\it Examples of curves with type $I_{\nu}^{*}$ reduction with $\nu$
odd
and 2-torsion points of singular reduction in and
not in the image of $\tau-1$}.
The curves $Y^{2}=X^{3}-23^{2}X+23^{3}$ and
$Y^{2}=X^{3}-23^{2}X-23^{3}$ have discriminant $\Delta = -2^{4}\cdot 23^{7}$.
Over ${\bf Q}_{23}$ they each have type $I_{\nu}^{*}$ reduction with $\nu =1$.
For both curves, two 2-torsion points are defined over a quadratic ramified
extension of ${\bf Q}_{23}$. Each curve has one 2-torsion point defined over
${\bf Q}_{23}$ with singular reduction. So the field $M_{23}$ is the quadratic
unramified extension of ${\bf Q}_{23}$.
For the first curve, the groups $E({\bf Q}_{23})/E_{0}({\bf Q}_{23})$ and
$E(M_{23})/E_{0}(M_{23})$ are both isomorphic to ${\bf Z}/4{\bf Z}$.
For the second curve,
the group $E({\bf Q}_{23})/E_{0}({\bf Q}_{23})$ is isomorphic to
${\bf Z}/2{\bf Z}$ and
$E(M_{23})/E_{0}(M_{23})$ is isomorphic to ${\bf Z}/4{\bf Z}$.
For both curves, $(\tau-1)
(E(M_{23})[2])$
is trivial and $E(M_{23})[2]$ has order 2. The action of $\tau$
on the
groups $E(M_{23})/E_{0}(M_{23})$ is identical to that in the last example.
So $(\tau-1)E(M_{23})\cap E(M_{23})[2]$ and $I^{2}({\bf Q}_{23},E)$ are
trivial for the first curve and have order 2 for the second curve.

\item {\it Examples of curves with type $I_{\nu}^{*}$ reduction with $\nu$
even
and 2-torsion points of singular reduction in and
not in the image of $\tau-1$}.
The curve $Y^{2}=X^{3}+X^{2}+4X+12$
has discriminant $\Delta =-2^{8}\cdot
3^{2}\cdot 23$. Over ${\bf Q}_{2}$ it has type $I_{\nu}^{*}$ reduction with
$\nu =0$.
All of the 2-torsion points are defined over ${\bf Q}_{2}$.
So the field $M_{2}$ is the quadratic unramified extension of ${\bf Q}_{2}$.
The group $(\tau-1)(E(M_{2})[2])$ is trivial.
Two points in $E(M_{2})[2]$
have non-singular reduction and so are in $(\tau-1)E(M_{2})$.
Two of the 2-torsion points
have singular reduction.
They generate $E({\bf Q}_{2})/E_{0}({\bf Q}_{2})$
where they have the same image; this group
has two elements.
The group $E({\bf Q}_{2}^{\rm unr})/E_{0}({\bf Q}_{2}^{\rm unr})$
is isomorphic to ${\bf Z}/2{\bf Z}\oplus {\bf Z}/2{\bf Z}$.
There are only 2 automorphisms
of ${\bf Z}/2{\bf Z}\oplus {\bf Z}/2{\bf Z}$ fixing a given
subgroup of order 2. So the group $E(M_{2})/E_{0}(M_{2})$
must be isomorphic to
${\bf Z}/2{\bf Z}\oplus
{\bf Z}/2{\bf Z}$.
The image of $\tau -1$ on an element of $E(M_{2})/E_{0}(M_{2})$ not fixed
by $\tau$ is the image of the 2 singular points in $E(M_{2})/E_{0}(M_{2})$.
Thus the two 2-torsion points with
singular reduction are also in $(\tau-1)E(M_{2})$, so $I^{2}({\bf Q}_{2},E)$
has 4 elements.

The curve $Y^{2}=X^{3}-25X$ has discriminant $\Delta =2^{6}\cdot 5^{6}$.
Over ${\bf Q}_{5}$ it also has type $I_{\nu}^{*}$ reduction with
$\nu =0$.
All of the 2-torsion points are defined over ${\bf Q}_{5}$.
This curve has three 2-torsion points with singular reduction.
The group $E({\bf Q}_{5})/E_{0}({\bf Q}_{5})$
is isomorphic to
${\bf Z}/2{\bf Z}\oplus
{\bf Z}/2{\bf Z}$ as it is generated by the 2-torsion points. Since
the group
$E({\bf Q}_{5}^{\rm unr})/E_{0}({\bf Q}_{5}^{\rm unr})$ has been
realized over ${\bf Q}_{5}$ it is impossible for these 2-torsion points to
be in the image of $\tau -1$. So the group $I^{2}({\bf Q}_{5},E)$ is trivial.

The curve $Y^{2}=X^{3}-75X+125$ has discriminant $\Delta =2^{4}\cdot
3^{4}\cdot 5^{6}$. Over ${\bf Q}_{5}$ it also has type $I_{\nu}^{*}$
reduction with $\nu =0$. This curve has three 2-torsion points with
singular reduction and they are all defined over $L_{5}$, the cubic unramified
extension of ${\bf Q}_{5}$. The group $E({\bf Q}_{5})/E_{0}({\bf Q}_{5})$
is trivial. The group $E(L_{5})/E_{0}(L_{5})$ is isomorphic to
${\bf Z}/2{\bf Z}\oplus{\bf Z}/2{\bf Z}$ and is generated by the 2-torsion
points.
The field $M_{5}$ is the sextic unramified extension
of ${\bf Q}_{5}$. The automorphism $\tau$ permutes the three 2-torsion points
with singular reduction.
Though all of the 2-torsion points are in
$(\tau -1)E(M_{5})$, they are also in $(\tau -1)E(M_{5})[2]$
and so the group $I^{2}({\bf Q}_{5},E)$
is trivial. This should be clear anyway as the group $E({\bf Q}_{5})[2]$
is trivial, so the group of unramified homomorphisms is also.
\end{enumerate}

\subsection{The groups $S^{2}(K,A)/I^{2}(K,A)$ and $C^{2}(K,A)/I^{2}(K,A)$}

Let us return to arbitrary abelian varieties and derive more elegant
upper bounds for the sizes of the quotients of the Selmer group and
the group related to the class group by their intersection, for the 2-map.
We can reformulate theorem~\ref{1st thm} with the following
divisibility conditions.
\begin{prop}
\label{divis}
We have
\[
\# S^{2}(K,A)/I^{2}(K,A)\;|\;
\prod\limits_{\Gp\in T}\frac{\# A(K_{\Gp})[2]}{\# I^{2}(K_{\Gp},A)}\]
and
\[
\# C^{2}(K,A)/I^{2}(K,A)\;|\;
\prod\limits_{\Gp |\Gf}\frac{\# A(K_{\Gp})[2]}{\# I^{2}(K_{\Gp},A)}\]
where $T$ is the set of primes of $K$ consisting of the
infinite primes
and the primes that divide $2\Gf$, where $\Gf$ is the conductor of $A$.
\end{prop}
\begin{proof}
The second divisibility statement follows immediately from
theorem~\ref{1st thm}. Let us consider the group
$S^{2}(K,A)/I^{2}(K,A)$.
The given divisibility statement differs from
what appears in
theorem~\ref{1st thm} at the primes dividing 2 and the infinite primes
since at all other primes we have $|\phi'(0)|=1$ as well as $A=A'$.
Let us consider the order of the group $A(K_{\Gp})/2A(K_{\Gp})$
where $\Gp$ is a prime of $K$, i.e.\ $S^{2}(K_{\Gp},A)$.
First assume that $\Gp$ is
a prime dividing 2. From proposition~\ref{pr:sizen},
we have
\[
\# A(K_{\Gp})/2A(K_{\Gp})=2^{g[K_{\Gp}:{\bf Q}_{2}]}\cdot\# A(K_{\Gp})[2]\]
where $g$ is the dimension of $A$.

If $K_{\Gp}\cong {\bf C}$, then
$A(K_{\Gp})/2A(K_{\Gp})$ is a trivial
group.
Let $K_{\Gp}\cong {\bf R}$.
We have the following commutative diagram from the snake lemma\[
\begin{array}{ccccccccc}
0 & \rightarrow & {\rm ker} & \rightarrow & A({\bf R})[2] & \rightarrow
& H^{1}(\Gal ({\bf C}/{\bf R}), \Lambda ) & \rightarrow & \\
 & & \downarrow & & \downarrow & & \downarrow & & \\
0 & \rightarrow
 & {\bf R}^{g}/\Lambda \cap({\bf R}^{g})
 & \rightarrow & A({\bf R}) & \rightarrow
 & H^{1}(\Gal ({\bf C}/{\bf R}), \Lambda ) & \rightarrow & 0 \\
 & & \;\;\;\;\;\downarrow [2] & & \;\;\;\;\;\downarrow [2]
 & & \;\;\;\;\;\downarrow [2] & & \\
0 & \rightarrow
 & {\bf R}^{g}/\Lambda \cap({\bf R}^{g})
 & \rightarrow & A({\bf R}) & \rightarrow
 & H^{1}(\Gal ({\bf C}/{\bf R}), \Lambda ) & \rightarrow & 0 \\
 & & \downarrow & & \downarrow & & \downarrow & & \\
 & \rightarrow & 0 & \rightarrow & A({\bf R})/2A({\bf R}) & \rightarrow
 & H^{1}(\Gal ({\bf C}/{\bf R}), \Lambda ) & \rightarrow & 0.
\end{array}\]
The center two short exact sequences appeared in the proof of
lemma~\ref{le:arch}.
The size of the kernel of the 2-map from ${\bf R}^{g}/\Lambda \cap
({\bf R}^{g})$
to itself is $2^{g}$.
We see that  the size of $A({\bf R})/2A({\bf R})$ equals
$\# A({\bf R})[2]/2^{g}$.
Notice that if $K_{\Gp}$ is isomorphic to ${\bf C}$ or to ${\bf R}$
that we have\[
\# A(K_{\Gp})/2A(K_{\Gp})=\#A(K_{\Gp})[2]/2^{g[K_{\Gp}:{\bf R}]}.\]
So the contribution from
the infinite
primes and the primes dividing 2 is the following
\[
\prod\limits_{\Gp |\infty}\frac{\# A(K_{\Gp})[2]}{2^{g[K_{\Gp}:{\bf R}]}}
\cdot\prod\limits_{\Gp |2}\frac{2^{g[K_{\Gp}:{\bf Q}_{2}]}
\cdot\#A(K_{\Gp})[2]}
{\# I^{2}(K_{\Gp},A)}\]
\[=\frac{2^{g[K:{\bf Q}]}}{2^{g[K:{\bf Q}]}}\cdot
\prod\limits_{\Gp |\infty}\# A(K_{\Gp})[2]
\cdot
\prod\limits_{\Gp |2}\frac{\# A(K_{\Gp})[2]}{\# I^{2}(K_{\Gp},A)}\]
\[ =
\prod\limits_{\Gp |2\atop {\rm or} \infty}\frac{\#
A(K_{\Gp})[2]}{\# I^{2}(K_{\Gp},A)}\]
since $I^{2}(K_{\Gp},A)$ is trivial for all infinite primes from
lemma~\ref{le:arch}.

\end{proof}

\subsection {Examples}

\noindent
Example I: {\em A cubic field whose class group has 2-rank at least 13}

\vspace{3mm}

Mestre \cite{M14} has produced an elliptic curve $E$ of rank at least 14
over ${\bf Q}$.  A Weierstrass equation for the curve $E$ is given by
\[Y^{2}+357573631Y=X^{3}+2597055X^{2}-549082X-19608054.\]
The discriminant is odd, square-free and negative. The primes of bad
reduction are those
that divide
the conductor of $E$ which are exactly the ones that divide the discriminant.
There are no non-trivial 2-torsion points defined over ${\bf Q}$.
Let us determine whether $L={\bf Q}(E[2])$ is an $A_{3}$- or an
$S_{3}$-extension of ${\bf Q}$.
Choose another Weierstrass equation for $E$ of the form $Y^{2}=g$
where $g\in {\bf Z}[X]$; the polynomial $g$ will be irreducible.
The quotient of the discriminants of the
two Weierstrass equations will be a square. The discriminant of the
second will be 16 times the discriminant of the cubic polynomial $g$.
If $g(\alpha)=0$, then $(\alpha ,0)$ is a 2-torsion point. The field
$F={\bf Q}(\alpha )$ is a cubic extension of ${\bf Q}$
whose discriminant differs
by a square from the discriminant of $g$. Thus the discriminant of the
field $F$ is not a square, so $L$ is an $S_{3}$-extension
of ${\bf Q}$ and $F$ is a non-Galois cubic extension of ${\bf Q}$.

Let $p$ be an odd prime of bad reduction. The
elliptic curve has type $I_{\nu}$ reduction (multiplicative)
at $p$ with $\nu =1$. So $E({\bf Q}_{p}^{\rm unr})/
E_{0}({\bf Q}_{p}^{\rm unr})$ is trivial and
the 2-torsion points defined over
unramified extensions of ${\bf Q}_{p}$ are in $E_{0}$.
From lemma~\ref{le:EmodE0}, all points of $E(M_{p})[2]$ are in
$(\tau-1)(E(M_{p}))$. From theorem~\ref{thm:mainthm}, the group
$I^{2}({\bf Q}_{p},E)$ has the same size as $E({\bf Q}_{p})[2]$. From
lemma~\ref{le:unr hom} and proposition~\ref{pr:sizen}, the groups
$C^{2}({\bf Q}_{p},E)$ and $S^{2}({\bf Q}_{p},E)$ also have the same
size as $E({\bf Q}_{p})[2]$ (which is 2). So
$C^{2}({\bf Q}_{p},E)$ and $S^{2}({\bf Q}_{p},E)$ are equal.

At the prime
2, the curve has supersingular good reduction and the reduced curve has
5 points over ${\bf F}_{2}$.
In fact the group $E({\bf Q}_{2})[2]$ is trivial and so the group of
unramified homomorphisms at the prime 2, namely $C^{2}({\bf Q}_{2},E)$,
is trivial and $E({\bf Q}_{2})/
2E({\bf Q}_{2})$, or $S^{2}({\bf Q}_{2},E)$, has 2 elements.
Since the discriminant is negative, we have $L_{\infty}={\bf R}(E[2])\cong
{\bf C}$ so from proposition~\ref{pr:archec}
the group $E({\bf R})/2E({\bf R})$ is trivial. So the three local groups
are all trivial for the infinite prime. Thus for all primes but 2, all
three local groups are the same. For the prime 2, the group $C^{2}(
{\bf Q}_{2},E)$ is contained in the group $S^{2}({\bf Q}_{2},E)$ with
index 2. Therefore,
the group of globally unramified homomorphisms, $C^{2}({\bf Q},E)$,
is contained in
the 2-Selmer group, $S^{2}({\bf Q},E)$,
with index at most 2.  To show it has index 2, we
need to find an element of the Selmer group that restricts to the
image of the non-trivial element of
$E({\bf Q}_{2})/2E({\bf Q}_{2})$ in $H^{1}({\bf Q}_{2},E[2])$.
There is a rational point with
$X$-coordinate $-2561042$. When multiplied by 5, one gets a point
in the kernel of reduction.
The valuation of the $X$-coordinate of this second point
at the prime 2 is $-2$ so it is in $E_{1}({\bf Q}_{2})$ but not in
$E_{2}({\bf Q}_{2})$. The group $E_{2}({\bf Q}_{2})$ is isomorphic to
the additive group of the ring of integers in ${\bf Q}_{2}$ and
$E_{1}({\bf Q}_{2})/E_{2}({\bf Q}_{2})$
has 2 elements (see \cite{Si}, chapter IV: proposition 3.2 and theorem 6.4).
Therefore this point cannot be in $2E({\bf Q}_{2})$.
This point
maps to a homomorphism ramified at 2.
So we know that $C^{2}({\bf Q},E)$ has rank at least 13. From
theorem~\ref{kernorm},
since ${\rm Cl}({\bf Q})$ is trivial, the group $C^{2}({\bf Q},E)$ is
isomorphic to the dual of
${\rm Cl}(F)/{\rm Cl}(F)^{2}$.
So using the generators that Mestre produced, we can show that the class group
of $F$ has 2-rank at least 13, but we cannot show it is any greater than
that.

\vspace{.2in}
\noindent
Example II: {\em A curve of genus 2 whose Jacobian has rank 6 or 7
over ${\bf Q}$.}

\vspace{3mm}

Let $C$ be the hyperelliptic curve defined over ${\bf Q}$ by the equation
$Y^{2}=f$ where
$f=X^{5}+16X^{4}-274X^{3}+817X^{2}+178X+1$.
Let $J$ be the Jacobian
of the normalization of $C$.
Let $\{ \alpha_{i}\}$ be the the set of roots of $f$. A basis for
$J[2]$ is the set of points
$(\alpha_{i},0)-\infty$ for $i=1$ to 4.
The point $(\alpha_{5},0)-\infty$ is the sum of the other 4.
The 2-torsion points are all defined over
the splitting field of $f$ which is the simplest quintic field $L$ with
discriminant $941^{4}$ (see \cite{Wa2}). Simplest quintic fields are
Galois over ${\bf Q}$.
The polynomial $f$
has discriminant $941^{4}191^{2}$.
The primes we have to worry about are $\infty , 2, 191$, and 941.
We would like to show that $J({\bf Q})$ has rank 6 or 7.

We have $L\cong F={\bf Q}[T]/f(T)$ in the notation of
subsection~\ref{1stsub}.
The group $H^{1}({\bf Q},J[2])$ is isomorphic to the kernel
of the norm from $L^{\ast}/L^{\ast 2}$ to ${\bf Q}^{\ast}/{\bf Q}^{\ast 2}$.
If we compose this isomorphism with the coboundary embedding
from $J({\bf Q})/2J({\bf Q})$ to $H^{1}({\bf Q},J[2])$ we get a map
which is identical to the map $X-T$ which is defined as follows (see
\cite{Sc}).
Pick a degree 0 divisor $D=\sum n_{i}R_{i}$
of $C$ defined over ${\bf Q}$ whose support
does not contain any point of $C$ with $Y=0$ or the point $\infty$.
Define\[
(X-T)(D)=\prod(X(R_{i})-T)^{n_{i}}.\]
One can show that
the image of a point $(x,y)-\infty$ where $y\neq 0$ is $x-T$
by finding the image of an
equivalent divisor whose support does not contain $\infty$.
Let $L_{p}={\bf Q}_{p}[T]/f(T)$. We similarly have embeddings of
$J({\bf Q}_{p})/2J({\bf Q}_{p})$ into $L_{p}^{\ast}/L_{p}^{\ast 2}$ by
the map $X-T$.
If the prime $p$ of ${\bf Q}$ splits in $L$ then $L_{p}$ is isomorphic
to the product of 5 copies of ${\bf Q}_{p}$ and the 2-torsion is rational
over ${\bf Q}_{p}$.
Then we
have \[
L_{p}\cong {\bf Q}_{p}[T]/f(T)\cong {\bf Q}_{p}\times\ldots\times
{\bf Q}_{p}\;\; {\rm by}\;\; T\mapsto (\alpha_{1},\ldots ,\alpha_{5}).\]
Again by finding the image of an equivalent divisor one can show that
the point $(\alpha_{i},0)-\infty$ maps to $\alpha_{i}-\alpha_{j}$ at the
$j$th component of $({\bf Q}_{p}^{\ast}/{\bf Q}_{p}^{\ast 2})^{5}$
for $j\neq i$ and to the product of the other 4 entries
at the $i$th component.

The 2-rank of
the class group of $L$ is 4 (see \cite{Wa2}).
Since condition $(5)$ of lemma~\ref{tfae} is satisfied, the groups
$H^{i}(G,J[2])$ are trivial for all $i$ and for all $G$ contained in
$\Gal (L/{\bf Q})$.
Therefore the group $H^{1}({\bf Q},J[2])$ is isomorphic
to both $\Hom_{\Gal (L/{\bf Q})}
(\Gal (\overline{L}/L),J[2])$ and to the kernel
of the norm from $L^{\ast}/L^{\ast 2}$ to ${\bf Q}^{\ast}/{\bf Q}^{\ast 2}$.
Thus, from theorem~\ref{kernorm}, the group
$C^{2}({\bf Q},J)$ is isomorphic to the group of unramified elements in
the kernel of the norm from
$L^{\ast}/L^{\ast 2}$ to ${\bf Q}^{\ast}/{\bf Q}^{\ast 2}$.
Since ${\bf Q}$ has a trivial class group, we know $C^{2}({\bf Q},J)$
is isomorphic to
exactly one
copy of the dual of ${\rm Cl}(L)/{\rm Cl}(L)^{2}$.
Therefore $C^{2}({\bf Q},J)$ has rank 4.
The roots of $f$ are all real, one is in the interval $(-27,-28)$,
two are in $(-1,0)$,
one in $(5,6)$ and one in $(6,7)$.
Since there are 3 negative roots and 2 positive roots and they are units,
the narrow and
wide class numbers of $L$ are the same.

The points $(-17,\pm 1223)$, $(-9,\pm 557)$, $(-6,\pm 317)$, $(-2,\pm 73)$,
$(0,\pm 1)$ and $(4,\pm 37)$ are all in
$C({\bf Q})$.
In this example, the notation $(n)$ will refer to
the divisor $(n,\sqrt{f(n)})-\infty$ and its image in
$J(N)/2J(N)$ over the appropriate field $N$.
Let us compute the local groups for the four interesting primes.
We know from the proof of proposition~\ref{divis} that the size of
$J({\bf R})/2J({\bf R})$
is the same as $\# J({\bf R})[2]/2^{2}$ which is 4.
If we map $(-2)$ into $(L_{\infty}^{\ast}/L_{\infty}^{\ast 2})^{5}$,
which is isomorphic to
$({\bf R}^{\ast}/{\bf R}^{\ast 2})^{5}$,
we get $(1,-1,-1,-1,-1)$ and if we map
$(0)$ we get $(1,1,1,-1,-1)$. So $(-2)$ and $(0)$ generate
$J({\bf R})/2J({\bf R})$.
Now the prime 2 is inert in $L$ and so $J({\bf Q}_{2})[2]$ is trivial.
Thus the group $C^{2}({\bf Q}_{2},J)$ is trivial and
from proposition~\ref{pr:sizen}, the group
$S^{2}({\bf Q}_{2},J)$
has rank 2. The prime 941 ramifies totally in $L$ and so
$J({\bf Q}_{941})[2]$ is trivial
as are the groups $S^{2}({\bf Q}_{941},J)$ and $C^{2}({\bf Q}_{941},J)$
(and so this prime can be ignored).

We have
\[
Y^{2}=f=(X-\alpha_{1})\ldots (X-\alpha_{5})
\equiv (X-5)(X-6)(X-37)(X-159)^{2}({\rm mod}\; 191).\]
So the prime 191 splits in $L$ and so $C^{2}({\bf Q}_{191},J)$ and
$S^{2}({\bf Q}_{191},J)$ each have rank 4.
In the following table we present a list of some rational points of
$J({\bf Q}_{191})$ and their images in
$L_{191}^{\ast}/L_{191}^{\ast 2}$
where $\pi$ is a prime element and $-1$ is a quadratic non-residue modulo
191.
Along the top is written $x-\alpha_{i}$
to remind us how to compute the $i$th component.

\[
\begin{array}{rrrrrrr}
& & x-5 & x-6 & x-37 & x-\alpha_{4} & x-\alpha_{5} \\
& (\alpha_{1}) & 1 & -1 & -1 & -1 & -1 \\
& (\alpha_{2}) & 1 & 1 & 1 & -1 & -1 \\
& (\alpha_{3}) & 1 & -1 & -1 & 1 & 1 \\
& (\alpha_{4}) & 1 & 1 & -1 & \pi & -\pi \\
& (\alpha_{5}) & 1 & 1 & -1 & \pi & -\pi \\
& (-17) & 1 & -1 & -1 & 1 & 1 \\
& (-9) &  1 & -1 & -1 & 1 & 1 \\
& (-6) &  1 & -1 & -1 & 1 & 1 \\
& (-2) &  1 & -1 & -1 & 1 & 1 \\
& (0) & -1 & -1 & 1 & 1 & 1 \\
& (4) & -1 & -1 & 1 & 1 & 1
\end{array}\]

We see that $(\alpha_{1}),(\alpha_{4}),(-2)$ and $(0)$ form a basis for
$J({\bf Q}_{191})/2J({\bf Q}_{191})$ and that \linebreak
$I^{2}({\bf Q}_{191},
J)$ has rank 3.
So our first estimates from theorem~\ref{1st thm}
show that $C^{2}({\bf Q},J)/
I^{2}({\bf Q},J)$ has rank at most 1, from the prime 191,
and $S^{2}({\bf Q},J)/I^{2}({\bf Q},J)$
has rank at most 5, which is
2 from the infinite prime, 2 from the prime 2 and 1 from
the prime 191. Using the fact that the
narrow and wide class numbers are the same, we see that we cannot
have a quadratic extension of $L$ which is ramified only at $\infty$
therefore $S^{2}({\bf Q},J)/I^{2}({\bf Q},J)$
has rank at most 3. Since $C^{2}({\bf Q},J)$
has rank 4, we see from theorem~\ref{1st thm}
that the rank of $S^{2}({\bf Q},J)$ is between 3 and 7.

At this point we will show that $(-17), (-9), (-6), (-2), (0)$ and $(4)$
are all independent points of infinite order in $J({\bf Q})$.
The prime 37 splits in $L$ and we have \[
Y^{2}=f\equiv (X-4)(X-8)(X-12)(X-16)(X-18)\;({\rm mod}\; 37).\]
The following shows the images of the given rational points in
$L_{37}^{\ast}/L_{37}^{\ast 2}$; where 2 is a quadratic non-residue of 37.
\[
\begin{array}{rrccccc}
& & x-4 & x-8 & x-12 & x-16 & x-18 \\
& (-17) & 1 & 1 & 2 & 1 & 2 \\
& (-9) & 2 & 2 & 1 & 1 & 1 \\
& (-6) & 1 & 2 & 2 & 2 & 2 \\
& (0) & 1 & 2 & 1 & 1 & 2 \\
& (-2) & 2 &  1 & 1 & 1 & 2 \\
& (4) & 1 & 1 & 2 & 1 & 2
\end{array}\]
Since none of the images of these points is trivial, none of them is a
torsion point.
Notice that $(-17), (-9), (-6), (0)$ are all independent, and so are
independent points of $J({\bf Q})$ of infinite order.
We have the relations $(-2)=(-9)+(-6)$ and $(4)=(-17)$ in $J({\bf Q}_{37})/
2J({\bf Q}_{37})$.

The prime 73 also splits in $L$ and we have\[
Y^{2}=f\equiv (X+26)(X+19)(X+2)(X-13)(X-18)\; ({\rm mod}\; 73).\]
The following shows the images of the given rational points in
$L_{73}^{\ast}/L_{73}^{\ast 2}$; where 5 is a quadratic non-residue of 73.
\[
\begin{array}{rrccccc}
& & x+26 & x+19 & x+2 & x-13 & x-18 \\
& (-17) & 1 & 1 & 5 & 5 & 1 \\
& (-9) & 5 & 5 & 5 & 5 & 1 \\
& (-6) & 5 & 5 & 1 & 1 & 1 \\
& (0) & 5 & 1 & 1 & 5 & 1 \\
& (-2) & 1 &  5 & 5 & 5 & 5 \\
& (4) & 5 & 1 & 1 & 1 & 5
\end{array}\]
We see that $(-2)$ is independent of $(-17),(-9),(-6),(0)$.
Now if there is a dependence relation of $(4)$ on the others in
$J({\bf Q})/2J({\bf Q})$ then the same relation holds in any
$J({\bf Q}_{p})/2J({\bf Q}_{p})$. Over ${\bf Q}_{37}$ the only relations
involving $(4)$ are $(4)=(-17)$ and $(4)=(-17)+(-2)+(-9)+(-6)$.
Over ${\bf Q}_{73}$ the only relations involving $(4)$ are
$(4)=(-9)+(-2)$ and $(4)=(-17)+(-6)+(-2)$. Since there is no intersection
in relations locally, there is no relation globally. So the 6 points
are independent and $J({\bf Q})$ has rank at least 6. If we could
show that $C^{2}({\bf Q},J)/I^{2}({\bf Q},J)$ had rank 1, then we would
know that $J({\bf Q})$ has rank exactly 6, but this is not the case.
In fact $C^{2}({\bf Q},J)=I^{2}({\bf Q},J)$.
It is a straightforward computation to show that the images of the
four independent points $(-2)+(-6),(-2)+(-9),(-2)+(-17),(0)+(4)$ are
unramified at all primes. So
$C^{2}({\bf Q},J)$ is contained
in $S^{2}({\bf Q},J)$ which has rank 6 or 7, as does the Mordell-Weil
group of $J$ over ${\bf Q}$.

\vspace{.2in}
\noindent
Example III: {\em $C^{2}(K,A)$ is not always contained in $S^{2}(K,A)$}

\vspace{3mm}

In the examples that have been worked out in the literature like the
ones above, the group of unramified homomorphisms is always contained
in the Selmer group; see \cite{Fr,Qu,Wa} and \cite{Si},
p.\ 321, 10.9(e).  This is not always the case, however.  Let $L$ be a
cyclic cubic extension of ${\bf Q}$ with class number 4; such fields
exist, like the one with conductor 163. Now $C^{2}({\bf Q},E)$ is
$\Hom_{\Gal (L/{\bf Q})}(\Gal (H(L)/L),E[2])$ where $H(L)$ is the
Hilbert class field of $L$. Since $H^{2}(\Gal (L/{\bf Q}), E[2])=0$, there
is only one possible Galois group $\Gal (H(L)/{\bf Q})$.  Since $A_{4}$
contains $V_{4}$ as a subgroup and the quotient acts on $V_{4}$ as
$\Gal (L/{\bf Q})$ acts on $E[2]$, we know that $A_{4}$ must be that group.
Let $\Gr$ be a prime of $H(L)$ whose restriction to ${\bf Q}$ is $p$,
an odd prime, and assume that $\Gr$ is an unramified extension of $p$
with a decomposition group of order two. By Tchebotarov's density
theorem, there are infinitely many such primes. Choose a polynomial
$f\in {\bf Z}[X]$ whose roots generate $L$ and assume $p$ does not
divide the discriminant of $f$. Let the following be the
factorization of $f$:
\[
f=(X-\alpha_{1})(X-\alpha_{2})(X-\alpha_{3})
\]
and define $g$ in the following way:
\[
g=(X-p\alpha_{1})(X-p\alpha_{2})(X-p\alpha_{3}).
\]
Let $E$ be the elliptic curve defined by the equation
$Y^{2}=g$. The points $T_{i}=(p\alpha_{i},0)$ are the non-trivial 2-torsion
points.
There are four $\Gal (L/{\bf Q})$-invariant homomorphisms from $\Gal (H(L)/L)$
to
$E[2]$; these are the four elements of $C^{2}({\bf Q},E)$.
They restrict onto the four unramified homomorphisms in
$\Hom(\Gal (\overline{\bf Q}_{p}/{\bf Q}_{p}),E[2])$.
The 2-torsion points are representatives for
$E({\bf Q}_{p})/2E({\bf Q}_{p})$. These points are mapped
to the homomorphisms that send an element $\psi \in \Gal
(\overline{\bf Q}_{p}/{\bf Q}_{p})$ to $\psi \frac{1}{2}T_{i}-\frac{1}{2}T_{i}$
where $\frac{1}{2}T_{i}$ is a 4-torsion point. The coordinates of the
4-torsion points lying over the point $(p\alpha_{1} ,0)$ are
\[
(p\alpha_{1} \pm p\sqrt{(\alpha_{1}-\alpha_{2})(\alpha_{1}-\alpha_{3})},
\pm p\sqrt{(\alpha_{1}-\alpha_{2})(\alpha_{1}-\alpha_{3})}
(\sqrt{p(\alpha_{1}-\alpha_{2})}\pm \sqrt{p(\alpha_{1}-\alpha_{3})})
\]
where the first and third $\pm$ must agree. From the $Y$-coordinates, it
is clear that the 4-torsion
points are defined over a ramified extension of ${\bf Q}_{p}$. By symmetry,
all three non-trivial 2-torsion points will map to ramified homomorphisms.
So the sizes of $S^{2}({\bf Q}_{p},E)$ and $C^{2}({\bf Q}_{p},E)$ are each 4
and they have trivial intersection.
One could also show this by noticing that $E$ has type $I_{0}^{\ast}$
reduction over ${\bf Q}_{p}$ with $E({\bf Q}_{p})/2E({\bf Q}_{p})\cong
{\bf Z}/2{\bf Z}\times {\bf Z}/2{\bf Z}$ and using theorem~\ref{thm:cases}.
Since
the group $C^{2}({\bf Q},E)$ maps onto the group $C^{2}({\bf Q}_{p},E)$,
which has trivial intersection with
the group $S^{2}({\bf Q}_{p},E)$,
the group $C^{2}({\bf Q},E)$ cannot
be contained in the 2-Selmer group.

\pagebreak

\pagebreak

\begin{center}
Edward F. Schaefer \\
Santa Clara University \\
Department of Mathematics\\
Santa Clara, CA  95053 \\
U.S.A. \\
eschaefer@scuacc.scu.edu
\end{center}

\end{document}